\documentclass{amsart}[12pt]
\usepackage{xcolor}
\definecolor{linkblue}{rgb}{0,0.2,0.6}
\def\docpdftitle{Invariant polynomials and Mukai models}

\usepackage[backrefs]{amsrefs}
\usepackage{mathrsfs}
\usepackage{enumitem}
\usepackage{tikz}
\usepackage[
	hyperindex,
	pdftex,
	pdftitle={\docpdftitle},
	pdfdisplaydoctitle,
	pdfpagemode=UseNone,
	breaklinks=true,
	extension=pdf,
	bookmarks=false,
	plainpages=false,
	colorlinks,
	linkcolor=linkblue,
	citecolor=linkblue,
	urlcolor=linkblue,
	pdfmenubar=true,
	pdftoolbar=true,
	pdfpagelabels,
	pdfpagelayout=SinglePage,
	pdfview=Fit,
	pdfstartview=Fit
]{hyperref}

\newtheorem{theorem}{Theorem}[section]
\newtheorem{lemma}[theorem]{Lemma}
\newtheorem{proposition}[theorem]{Proposition}

\theoremstyle{definition}
\newtheorem{definition}[theorem]{Definition}

\numberwithin{equation}{section}

\newcommand{\Sym}{\operatorname{Sym}}

\newcommand{\SL}{\operatorname{SL}}

\newcommand{\OG}{\operatorname{OG}(5,10)}

\newcommand{\Sp}{\operatorname{Sp}}
\newcommand{\SpGr}{\operatorname{SpGr}(3,6)}
\newcommand{\Spin}{\operatorname{Spin}}
\newcommand{\Gr}{\operatorname{Gr}}

\newcommand{\git}{/\!\!/}
\newcommand{\M}{\overline{M}}
\newcommand{\mywedge}{\mathsf{\Lambda}}

\newcommand{\mukai}{\mathrm{MX}}
\newcommand{\LieG}{G}
\newcommand{\AffX}{\Gamma}

\newcommand{\GT}{\operatorname{GT}}
\newcommand{\wt}{\operatorname{wt}}

\begin{document}

\title{Invariant polynomials and Mukai's models of moduli spaces of
  curves and K3 surfaces}


\author{David Swinarski}
\address{Department of Mathematics\\ 
Fordham University\\ 
441 E Fordham Rd\\ 
Bronx, NY 10458}
\email{dswinarski@fordham.edu}


\date{}

\begin{abstract}
Beginning in the 1980s, Mukai introduced birational models of some
moduli spaces of curves and some moduli spaces of K3 surfaces. They
are defined as geometric invariant theory quotients. Little is known
about the boundaries of these spaces. We describe an approach to
efficiently evaluate certain invariant polynomials associated to these
GIT problems. This allows us to show that several singular curves and
surfaces are GIT semistable in Mukai's models. In the appendix, we
give a combinatorial formula for an $\SL_n$-invariant in terms of the
Gelfand-Tsetlin basis. 
\end{abstract}

\maketitle

\section{Introduction}

In a series of papers beginning in 1988, Mukai introduced birational
models $\mukai^{1}_g$ of some moduli spaces
of curves, and birational models $\mukai^{2}_g$ of some moduli spaces
of K3 surfaces
\cites{Mukai1988,Mukai1992a,Mukai1992b,Mukai1995,Mukai2010}\footnote{There
is no standard notation in the literature for Mukai's models. We
propose $\mukai^{1}_g$ and $\mukai^{2}_g$, where the
$\mathrm{M}$ and $\mathrm{X}$ are read as upper case Greek letters:
Mu-Chi.}. For curves, Mukai's models $\mukai^{1}_g$ are
birational to $\overline{M}_g$, the moduli space of Deligne-Mumford
stable curves of genus $g$, for $g=7,8,9$. For genus $10$, Mukai also
gives a model of the moduli space of curves that lie on a K3 surface;
these form a divisor in $\overline{M}_{10}$.  For surfaces, Mukai's
models $\mukai^{2}_g$ 
are models of the moduli spaces of polarized K3 surfaces of genus
$g$ for $g=7,8,9,10$.

Each of Mukai's models is defined as a geometric invariant theory
(GIT) quotient. These descriptions will be given in Section \ref{sec:
  GIT problems}. What are the GIT stable, semistable, and unstable loci for these GIT problems? 

A few foundational results are known for these GIT problems. For K3 surfaces, Greer,
Li, and Tian prove that smooth, Brill-Noether general K3 surfaces are
GIT stable in Mukai's models \cite{GLT}*{Thm.~2.9}. For curves, Farkas and Verra prove
that a general smooth or irreducible nodal curve with at most $g-1$
nodes is GIT semistable in a GIT problem that is closely related to
Mukai's model \cite{FV}*{Prop.~3.5}; as a corollary, the same statement is true for Mukai's
models.  However, many questions remain open regarding which
objects may appear on the boundary of Mukai's models.

In \cite{Swinarski2025} the author gave examples of three singular
curves that appear in Mukai's model of $\overline{M}_7$. 
\begin{center}
\begin{tabular}{ll}
Genus 7, Example 1: & the $7$-cuspidal curve with heptagonal symmetry\\
Genus 7, Example 2: & the balanced ribbon of genus 7\\
Genus 7, Example 3: & a family of genus 7 graph curves
\end{tabular}
\end{center}

The proofs of semistability in \cite{Swinarski2025} are classical and
explicit: find an invariant polynomial $F$ of Mukai’s GIT problem and
a point $[P]$ parameterizing the curve of interest and show that
$F([P]) \neq 0$ by direct calculation.

In this work, we improve and expand on the ideas introduced in
\cite{Swinarski2025}. Here, we describe a different method for
constructing and evaluating invariant polynomials that is vastly more
efficient than the approach used in \cite{Swinarski2025}. We
implement this method in \texttt{Macaulay2}, and use our code to show that several
examples in each of Mukai's models are GIT semistable. This yields the
following theorem. 

\begin{theorem} \label{thm:main}
Each of the following curves or surfaces is GIT semistable in the corresponding Mukai
model.

Genus 7:
\begin{enumerate}
  \item[(a)] The genus 7 tangent developable
  \item[(b)] The genus 7 cuspidal cubic with 7-gonal symmetry
  \item[(c)] The genus 7 balanced K3 carpet
  \item[(d)] The genus 7 balanced ribbon
  \item[(e)] The genus 7 graph curve for the graph $\Gamma_7$
\end{enumerate}

Genus 8: 
\begin{enumerate}
  \item[(a)] The genus 8 tangent developable
  \item[(b)] The genus 8 cuspidal cubic with 8-gonal symmetry
  \item[(c)] A genus 8 reducible surface with four components, each
   generically nonreduced
  \item[(d)] A genus 8 reducible curve with five components, each
  generically nonreduced
  \item[(e)] The genus 8 graph curve for the graph $\Gamma_8$
\end{enumerate}

Genus 9:
\begin{enumerate}
  \item[(a)] The genus 9 tangent developable
  \item[(b)] The genus 9 cuspidal cubic with 9-gonal symmetry   
  \item[(c)] The genus 9 balanced K3 carpet
  \item[(d)] The genus 9 balanced ribbon
  \item[(e)] The genus 9 graph curve for the graph $\Gamma_9$
\end{enumerate}

Genus 10:
\begin{enumerate}
  \item[(a)] The genus 10 tangent developable
  \item[(b)] The genus 10 cuspidal cubic with 10-gonal symmetry
  \item[(c)] A genus 10 reducible surface with two components, each generically nonreduced
  \item[(d)] A genus 10 reducible curve with four components, each generically nonreduced
\end{enumerate}

The graphs $\Gamma_7$, $\Gamma_8$, and $\Gamma_9$ are shown in  Figure
\ref{fig:ThreeGraphs}. See Theorem \ref{thm: points in parameter
  space} for the equations defining examples 8c, 8d, 10c, and 10d.   

\end{theorem}

Here are the trivalent graphs whose graph curves we study. The genus 8
graph below is known as the Heawood graph. 

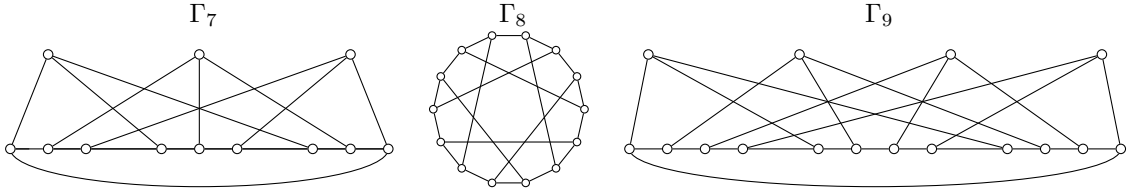
\begin{figure}[h]
      \caption{Three trivalent graphs}
        \label{fig:ThreeGraphs}    
  \begin{center}
    \begin{tabular}{ccc}
$\Gamma_7 $ & $\Gamma_8$ & $\Gamma_9$ \\
\begin{tikzpicture}[scale=0.25]
  \draw (0,0)--(1,0); 
  \draw (0,0) arc (180:360:10 and 2);
  \draw (0,0)--(20,0); 
  \draw (0,0)--(2,5); 
  \draw (2,0)--(4,0); 
  \draw (2,0)--(10,5); 
  \draw (4,0) -- (8,0); 
  \draw (4,0)--(18,5); 
  \draw (8,0)--(10,0); 
  \draw (8,0)--(2,5); 
  \draw (10,0)--(12,0); 
  \draw (10,0)--(10,5); 
  \draw (12,0) -- (16,0); 
  \draw (12,0)--(18,5); 
  \draw (16,0)--(18,0); 
  \draw (16,0)--(2,5); 
  \draw (18,0)--(20,0); 
  \draw (18,0)--(10,5); 
  \draw (20,0)--(18,5); 
  \filldraw[white] (0,0) circle (0.25);
  \filldraw[white] (2,0) circle (0.25);
  \filldraw[white] (4,0) circle (0.25);
  \filldraw[white] (8,0) circle (0.25);
  \filldraw[white] (10,0) circle (0.25);
  \filldraw[white] (12,0) circle (0.25);
  \filldraw[white] (16,0) circle (0.25);
  \filldraw[white] (18,0) circle (0.25);
  \filldraw[white] (20,0) circle (0.25);
  \filldraw[white] (2,5) circle (0.25);
  \filldraw[white] (10,5) circle (0.25);
  \filldraw[white] (18,5) circle (0.25);
  \draw (0,0) circle (0.25);
  \draw (2,0) circle (0.25);
  \draw (4,0) circle (0.25);
  \draw (8,0) circle (0.25);
  \draw (10,0) circle (0.25);
  \draw (12,0) circle (0.25);
  \draw (16,0) circle (0.25);
  \draw (18,0) circle (0.25);
  \draw (20,0) circle (0.25);
  \draw (2,5) circle (0.25);
  \draw (10,5) circle (0.25);
  \draw (18,5) circle (0.25);    
\end{tikzpicture} &     \begin{tikzpicture}[scale=1]
  \draw (1,0)--(0.900969, 0.433884)--(0.62349, 0.781831)--(0.222521, 0.974928)--(-0.222521,0.974928)--(-0.62349, 0.781831)--(-0.900969, 0.433884)--(-1,0)--(-0.900969,-0.433884)--(-0.62349, -0.781831)--(-0.222521, -0.974928)--(0.222521, -0.974928)--(0.62349,-0.781831)--(0.900969, -0.433884)--cycle;
  \draw (1,0)--(-0.62349, 0.781831); 
  \draw (0.62349, 0.781831)--(-1,0); 
  \draw (-0.222521,0.974928)--(-0.62349, -0.781831); 
  \draw (-0.900969, 0.433884)--(0.222521, -0.974928); 
  \draw (-0.900969,-0.433884)--(0.900969, -0.433884); 
  \draw (-0.222521, -0.974928)--(0.900969, 0.433884); 
  \draw (0.62349,-0.781831)--(0.222521, 0.974928); 
  \filldraw[white] (1,0) circle (0.05); 
  \filldraw[white] (0.900969, 0.433884) circle (0.05); 
  \filldraw[white] (0.62349, 0.781831) circle (0.05); 
  \filldraw[white] (0.222521, 0.974928) circle (0.05); 
  \filldraw[white] (-0.222521,0.974928) circle (0.05); 
  \filldraw[white] (-0.62349, 0.781831) circle (0.05); 
  \filldraw[white] (-0.900969, 0.433884) circle (0.05); 
  \filldraw[white] (-1,0) circle (0.05); 
  \filldraw[white] (-0.900969,-0.433884) circle (0.05); 
  \filldraw[white] (-0.62349, -0.781831) circle (0.05); 
  \filldraw[white] (-0.222521, -0.974928) circle (0.05); 
  \filldraw[white] (0.222521, -0.974928) circle (0.05); 
  \filldraw[white] (0.62349,-0.781831) circle (0.05); 
  \filldraw[white] (0.900969, -0.433884) circle (0.05); 
  \draw (1,0) circle (0.05); 
  \draw (0.900969, 0.433884) circle (0.05); 
  \draw (0.62349, 0.781831) circle (0.05); 
  \draw (0.222521, 0.974928) circle (0.05); 
  \draw (-0.222521,0.974928) circle (0.05); 
  \draw (-0.62349, 0.781831) circle (0.05); 
  \draw (-0.900969, 0.433884) circle (0.05); 
  \draw (-1,0) circle (0.05); 
  \draw (-0.900969,-0.433884) circle (0.05); 
  \draw (-0.62349, -0.781831) circle (0.05); 
  \draw (-0.222521, -0.974928) circle (0.05); 
  \draw (0.222521, -0.974928) circle (0.05); 
  \draw (0.62349,-0.781831) circle (0.05); 
  \draw (0.900969, -0.433884) circle (0.05); 
\end{tikzpicture} & \begin{tikzpicture}[scale=0.25]
  \draw (0,0)--(2,0); 
  \draw (2,0)--(4,0); 
  \draw (4,0)--(6,0); 
  \draw (6,0)--(10,0); 
  \draw (10,0)--(12,0); 
  \draw (12,0)--(14,0); 
  \draw (14,0)--(16,0); 
  \draw (16,0)--(20,0); 
  \draw (20,0)--(22,0); 
  \draw (22,0)--(24,0); 
  \draw (24,0)--(26,0); 
  \draw (0,0) arc (180:360:13 and 2); 
  \draw (0,0)--(1,5); 
  \draw (10,0)--(1,5); 
  \draw (20,0)--(1,5); 
  \draw (2,0)--(9,5); 
  \draw (12,0)--(9,5); 
  \draw (22,0)--(9,5); 
  \draw (4,0)--(17,5); 
  \draw (14,0)--(17,5); 
  \draw (24,0)--(17,5); 
  \draw (6,0)--(25,5); 
  \draw (16,0)--(25,5); 
  \draw (26,0)--(25,5); 
  \filldraw[white] (0,0) circle (0.25); 
  \filldraw[white] (2,0) circle (0.25); 
  \filldraw[white] (4,0) circle (0.25); 
  \filldraw[white] (6,0) circle (0.25); 
  \filldraw[white] (10,0) circle (0.25); 
  \filldraw[white] (12,0) circle (0.25); 
  \filldraw[white] (14,0) circle (0.25); 
  \filldraw[white] (16,0) circle (0.25); 
  \filldraw[white] (20,0) circle (0.25); 
  \filldraw[white] (22,0) circle (0.25); 
  \filldraw[white] (24,0) circle (0.25); 
  \filldraw[white] (26,0) circle (0.25); 
  \filldraw[white] (1,5) circle (0.25); 
  \filldraw[white] (9,5) circle (0.25); 
  \filldraw[white] (17,5) circle (0.25); 
  \filldraw[white] (25,5) circle (0.25); 
  \draw (0,0) circle (0.25); 
  \draw (2,0) circle (0.25); 
  \draw (4,0) circle (0.25); 
  \draw (6,0) circle (0.25); 
  \draw (10,0) circle (0.25); 
  \draw (12,0) circle (0.25); 
  \draw (14,0) circle (0.25); 
  \draw (16,0) circle (0.25); 
  \draw (20,0) circle (0.25); 
  \draw (22,0) circle (0.25); 
  \draw (24,0) circle (0.25); 
  \draw (26,0) circle (0.25); 
  \draw (1,5) circle (0.25); 
  \draw (9,5) circle (0.25); 
  \draw (17,5) circle (0.25); 
  \draw (25,5) circle (0.25); 
 \end{tikzpicture}
    \end{tabular}
  \end{center}
  \end{figure}

The examples for genus $g=8,9,10$ were selected as analogues of the genus 7 examples that appeared in 
\cite{Swinarski2025}. The genus 7 examples were selected because they
appear in the literature related to Green's Conjecture and the
Hassett-Keel program. See \cite{Swinarski2025}*{Section 1.1} for more
details and references. However, in our opinion, the most striking
features of the list appearing in Theorem \ref{thm:main} are its
length and breadth, rather than the geometric properties of any
particular items on the list. It is clear that one could use the
methods and code described here to analyze many more examples for these GIT problems.

\subsection*{Outline}
In Section \ref{sec: GIT problems} we describe the GIT problems used
to define each of Mukai's models. In Section \ref{sec: invariant
  polynomials} we describe the invariant polynomials $F_{g,d}$ up to
scaling that we use to prove the main theorem.  In Section \ref{sec:
  finding linear spaces} we describe the points in Mukai's parameter
spaces that yield the examples listed in the main theorem. We also
describe in detail how we found these points for two examples (the
balanced K3 carpet and balanced ribbon in genus 9). In Section
\ref{sec: proof} we prove the main theorem and make some remarks on the computer
calculations involved.

Finally, in the self-contained Appendix, we
prove a combinatorial formula for an $\SL_n$-invariant in the
representation $V(\lambda)\otimes V(\lambda^{*})$ in terms the
Gelfand-Tsetlin bases of $V(\lambda)$ and $V(\lambda^{*})$. This formula is
implemented in the \texttt{LieAlgebraRepresentations} package and used once in the main part of the paper to give a fast
calculation of one step in constructing the polynomial
$F_{8,2}$. This material is presented in an appendix because it is thematically distinct from the rest of the paper.

\subsection*{Acknowledgements}
It is a pleasure to thank  Ian
Morrison, Han-Bom Moon, and Paul Zinn-Justin for many helpful discussions. This work was partially supported by a
Fordham University Faculty Fellowship and a Fordham University Mid-Career Micro-Grant.

\subsection*{Software and code links}
The calculations reported here were performed in
\texttt{Macaulay2} version 1.25.11 \cite{M2}. Many of these use the new
\texttt{LieAlgebraRepresentations} package by the author and Paul Zinn-Justin. In developing the package and the code for this project, we also used 
\texttt{GAP}, \texttt{Magma}, and \texttt{Sage} \cites{GAP, Magma, Sage}.

The input and output files used for the calculations are posted on
the author's webpage \cite{Code}.

\section{The GIT problems associated to Mukai's models } \label{sec: GIT problems}

Each Mukai model is the GIT quotient of a Grassmannian variety by a
simple linear algebraic group. We use the following notation.

\begin{definition} \label{def: Vc notation}
  Let $\mathfrak{g}$ be a simple Lie algebra of rank $n$. Let
  $\omega_1,\ldots,\omega_n$ be the fundamental dominant weights for
  $\mathfrak{g}$, and let $\lambda = \sum c_i \omega_i$ be a dominant
  integral weight. We write $V(c_1,\ldots,c_n)$ for the irreducible
  $\mathfrak{g}$-module with highest weight $\lambda$. 
\end{definition}

\textit{Apology.} We
apologize for the notation $\LieG_g$ used below. On the one hand, this is natural,
since we have a sequence of groups $G$ in each genus $g=7,8,9,10$. On the
other hand, it is terrible because one of these groups is the
exceptional group $G_2$. We trust that the reader can muddle through. 

\begin{definition}
  We define Mukai models for $g=7,8,9,10$ as the following GIT quotients.
\begin{align*}
  \mukai^{1}_g &:= \Gr(g,V_g) \git \LieG_g \\
  \mukai^{2}_g &:= \Gr(g+1,V_g) \git \LieG_g
\end{align*}
Here $V_g$ and $\LieG_g$ are given by the following table.
\begin{center}
\begin{tabular}{llll}
  $g$  & $\LieG_g$ & $V_g$ & Description of $V_g$ \\
  7 & $\Spin(10)$ & $V(0,0,0,1,0)$ & half spin representation \\
  8 & $\SL(6)$ & $V(0,1,0,0,0)$ & $\mywedge^2 \mathbb{C}^6$ \\
  9 & $\Sp(6)$ & $V(0,0,1)$ & $\ker(\mywedge^3 \mathbb{C}^6 \rightarrow \mathbb{C}^6 )$\\
  10 & $G_2$ & $V(0,1)$ & adjoint representation
\end{tabular}
\end{center}
\end{definition}

A key role is played by the following homogeneous varieties $X_g$,
which are the orbits of a highest weight vector in the representations $V_g$. 

\begin{center}
\begin{tabular}{lll}
  $g$ & $X_g$ & Description of $X_g$ \\
  7 & $\OG$ & The orthogonal Grassmannian, or isotropic Grassmannian\\
                    && Also known as the 10-dimensional spinor variety \\
  8 & $\Gr(2,6)$ & \\
  9 & $\SpGr$ & The symplectic Grassmannian, or Lagrangian Grassmannian\\
  10 & & The adjoint variety of $G_2$
\end{tabular}
\end{center}

We obtain curves from points in the parameter spaces of
Mukai's models as follows. Let $[P] \in \Gr(g,V_g)$ be general. Then $P \cap
X_g$ is a smooth canonically embedded curve of genus $g$. (A similar
statement is true for surfaces, replacing $\Gr(g,V_g)$ by $\Gr(g+1,V_g)$.)

\section{Invariant polynomials associated to Mukai's
  models} \label{sec: invariant polynomials}

We have constructed one or more invariant polynomials for each of
Mukai's models. Here, for each Mukai model $\mukai^{k}_g$, we describe
an invariant polynomial $F_{g,k}$ that establishes GIT semistability
for our selected examples. We repeatedly
exploit the fact that the representations $\mywedge^{g} V_g^{*}$ and
$\mywedge^{g+1} V_g^{*}$ are reducible. This permits us to efficiently evaluate
polynomials that are suitably adapted to these decompositions.

\subsection{How we selected these polynomials}
Suppose that $\mywedge^{g} V_g^{*}$ decomposes into irreducibles
$\oplus V(\lambda_i)^{m_i}$. We typically select the irreducible $V(\lambda_i)$ with smallest dimension. Then we seek the lowest degree $d$ such that $\Sym^d V(\lambda_i)$ contains an invariant. If $d=4,6$ we may seek trivial $G$-modules among the products of submodules of $\Sym^2 V(\lambda_i)$. When doing so, we typically select submodules of small dimension, but not necessarily the smallest ones.

\subsection{Genus 7, dimension 1}
Mukai's model is
\[
\mukai^{1}_7 =  \Gr(7,V_7) \git \LieG_7 = \Gr(7,V(0,0,0,1,0)) \git \Spin(10).
\]
Here $V_7$ is the 16-dimensional half-spin representation of
$\Spin(10)$. The corresponding root system D5 has rank 5. The highest
weight of $V_7$ is the fundamental dominant weight
$\omega_4$. Thus, using the notation $V(c_1,c_2,c_3,c_4,c_5)$ from
Definition \ref{def: Vc notation}, we write $V_7 = V(0,0,0,1,0)$.  

A character calculation shows that the lowest degree invariants for $\mywedge^7 V(0,0,0,1,0)^{*} $ are in
degree 4.

We have the following decomposition. Below each representation, we
record its dimension.
\[
\begin{array}{ccccc}
  \mywedge^7 V(0,0,0,1,0)^{*} & = & V(1,0,1,1,0) & \oplus & V(3,0,0,0,1)
  \\
  11440 & = & 8800 &+ & 2640 
\end{array}
\]

Furthermore, $\Sym^2 V(3,0,0,0,1)$ contains $V(3,0,0,0,0)$ with
multiplicity 1, and $\Sym^2 V(3,0,0,0,0)$ contains the trivial representation with
multiplicity 1. This defines $F_{7,1} \in
\Sym^4 \mywedge^7 V_7^{*}$ up to scaling. (Note: a different invariant
polynomial $F_{5\omega_1}$ was used in \cite{Swinarski2025}. The method for evaluating
$F_{7,1}$ here is far more efficient than the method used to
evaluate $F_{5 \omega_1}$ there.)

\subsection{Genus 7, dimension 2}
Mukai's model is
\[
\mukai^{2}_7 =  \Gr(8,V_7) \git \LieG_7 = \Gr(8,V(0,0,0,1,0)) \git \Spin(10).
\]
See the previous subsection for some additional comments on this
notation.

A character calculation shows that the lowest degree invariants are in
degree 2.

We have the following decomposition. Below each representation, we
record its dimension.
\[
\begin{array}{ccccccc}
  \mywedge^8 V(0,0,0,1,0)^{*} & = & V(0,0,2,0,0) & \oplus & V(2,0,0,1,1) & \oplus &V(4,0,0,0,0)
  \\
  12870 & = & 4125&+ & 8085 & + & 660
\end{array}
\]

Here, $\Sym^2 V(4,0,0,0,0)$ contains the trivial representation with
multiplicity 1. This defines $F_{7,2} \in
\Sym^2 \mywedge^8 V_7^{*}$ up to scaling. 

\subsection{Genus 8, dimension 1}
Mukai's model is
\[
\mukai^{1}_8=  \Gr(8,V_8) \git \LieG_8 = \Gr(8,V(0,1,0,0,0)) \git \SL(6).
\]
Here $V_8$ is the 15-dimensional representation $\mywedge^2
\mathbb{C}^6$. The corresponding root system A5 has rank 5. The highest
weight of $V_8$ is the fundamental dominant weight
$\omega_2$. Thus, using the notation $V(c_1,c_2,c_3,c_4,c_5)$ from
Definition \ref{def: Vc notation}, we write $V_8 = V(0,1,0,0,0)$.  In
type A, it is also common to describe irreducible representations
using partitions. The partition corresponding to $V_8$ is $1,1,0,0,0,0$.

A character calculation shows that the lowest degree invariants for $\mywedge^8 V(0,1,0,0,0)^{*}$ are in
degree 3. However, our first attempts at evaluating degree 3 invariants
suggest that some or all of them are in the ideal generated by the Pl\"{u}cker
relations; for instance, they could be linear syzygies among some of these
quadratic relations. Therefore, we looked for degree 6 invariant
polynomials instead. 

We have the following decomposition. Below each representation, we
record its dimension, and the partitions corresponding to the
irreducible representations.
\[
\begin{array}{ccccccc}
  \mywedge^8 V(0,1,0,0,0)^{*} & = & V(2,1,0,1,0) & \oplus & V(0,2,0,0,2)& \oplus & V(1,0,0,2,1)
  \\
  6435 & = & 2430 &+ & 1800 &+ & 2205 \\
           &     & 4,2,1,1,0,0 & &4,2,2,2,0,0 & &4,3,3,3,1,0
\end{array}
\]

We study the smallest of these submodules. $\Sym^2  V(0,2,0,0,2)$
contains a summand $V(0,2,0,0,0) $ with multiplicity 1. $V(0,2,0,0,0)$
has dimension 105, and corresponds to the partition $2,2,0,0,0,0$. We
find that $\Sym^3 V(0,2,0,0,0)$ contains the trivial representation with multiplicity 1. This defines $F_{8,1} \in
\Sym^6 \mywedge^8 V_8^{*}$ up to scaling.

\subsection{Genus 8, dimension 2} \label{subsec: genus 8 dim 2}
Mukai's model is
\[
\mukai^{2}_8 =  \Gr(9,V_8) \git \LieG_8 = \Gr(9,V(0,1,0,0,0)) \git \SL(6).
\]
See the previous subsection for some additional comments on this
notation.

A character calculation shows that the lowest degree invariants are in
degree 2. Here is a description.

We have the following decomposition. Below each representation, we
record its dimension and the corresponding partition.
\[
\begin{array}{ccccccc}
  \mywedge^9 V(0,1,0,0,0)^{*} & = & V(1,1,0,1,1) & \oplus & V(3,0,1,0,0) & \oplus & V(0,0,0,3,0)
  \\
  5005 & = & 3675 &+ & 840 &+ & 490 \\
           &    & 4,3,3,2,1,0 && 4,1,1,0,0,0 & & 3,3,3,3,0,0
\end{array}
\]

The first summand $V(1,1,0,1,1)$ is self-dual. Recall that
$\dim (V(\lambda) \otimes V(\lambda^{*}))^G=1$. Moreover, in type A, we have a
combinatorial formula for this invariant polynomial in terms of the
Gelfand-Tsetlin bases of $V(\lambda)$ and $V(\lambda^{*})$; see Theorem \ref{theorem: invariant polynomial formula} in the Appendix. Applying this to $\lambda =
(1,1,0,1,1)$ yields the invariant polynomial $F_{8,2} \in \Sym^2 \mywedge^9 V_8^{*}$ up to scaling.

\subsection{Genus 9, dimension 1}
Mukai's model is
\[
\mukai^{1}_9 =  \Gr(9,V_9) \git \LieG_9 = \Gr(6,V(0,0,1)) \git \Sp(6).
\]
Here $V_9$ is the 14-dimensional irreducible representation of 
$\Sp(6)$. See \cite{FH}*{Section 17.2} for a construction of $V_9$ as
the kernel of an explicit contraction map. The highest
weight of $V_9$ is the fundamental dominant weight
$\omega_3$. Thus, using the notation $V(c_1,c_2,c_3)$ from
Definition \ref{def: Vc notation}, we write $V_9 = V(0,0,1)$.  

A character calculation shows that the lowest degree invariants for $\mywedge^9 V(0,0,1)^{*}$ are in
degree 4.

We have the following decomposition. Below each representation, we
record its dimension.
\[
\begin{array}{ccccccccc}
  \mywedge^9 V(0,0,1)^{*} & = & V(0,0,1) & \oplus & V(1,2,0) & \oplus & V(5,0,0) & \oplus & V(2,1,1)
  \\
  2002 & = & 14 &+ & 350 &+ & 252 &+ & 1386 
\end{array}
\]


$\Sym^2 V(5,0,0)$ contains $V(2,0,0)$ with
multiplicity 1, and $\Sym^2 V(2,0,0)$ contains the trivial
representation with multiplicity 1. This defines $F_{9,1} \in
\Sym^4 \mywedge^9 V_9^{*}$ up to scaling.

\subsection{Genus 9, dimension 2}
Mukai's model is
\[
\mukai^{2}_9 =  \Gr(10,V_9) \git \LieG_9 = \Gr(10,V(0,0,1)) \git \Sp(6).
\]
See the previous subsection for some additional comments on this
notation.

Here, $\mywedge^{10} V(0,0,1)^{*}$ contains the trivial representation with multiplicity 1. This defines $F_{9,2} \in
\Sym^1 \mywedge^{10} V_9^{*}$ up to scaling.

\subsection{Genus 10, dimension 1}
Mukai's model is
\[
\mukai^{1}_{10} =  \Gr(10,V_{10}) \git \LieG_{10} = \Gr(10,V(0,1)) \git G_2.
\]
(Again, we apologize for the notation: $\LieG_{10}$, the group acting on Mukai's models for genus 10, is the exceptional group $G_2$.)

Here $V_{10}$ is the adjoint representation of $G_2$. The highest
weight of $V_{10}$ is the fundamental dominant weight
$\omega_2$. Thus, using the notation $V(c_1,c_2)$ from
Definition \ref{def: Vc notation}, we write $V_{10} = V(0,1)$.  

A character calculation shows that the lowest degree invariants are in
degree 2.

The representation $\mywedge^{10} V(0,1)^{*} $ decomposes as a sum of
seven irreducibles, each with multiplicity 1. One of these summands is
(a copy of) $V(0,1)$. Furthermore
$\Sym^2 V(0,1)$ contains the trivial representation with multiplicity 1. This defines $F_{10,1} \in
\Sym^2 \mywedge^{10} V_{10}^{*}$ up to scaling.

\subsection{Genus 10, dimension 2}
Mukai's model is
\[
\mukai^{2}_{10} =  \Gr(11,V_{10}) \git \LieG_{10} = \Gr(11,V(0,1)) \git G_2.
\]
See the previous subsection for some additional comments on this
notation.

Here, $\mywedge^{11} V(0,1)^{*}$ contains the trivial representation with multiplicity 1. This defines $F_{10,2} \in
\Sym^1 \mywedge^{11} V_{10}^{*}$ up to scaling.

\section{Locating specific singular curves and surfaces in the parameter spaces of Mukai's models}
\label{sec: finding linear spaces}
Our approach to proving the Main Theorem requires that we find points in the parameter spaces of Mukai's models corresponding to the singular curves and surfaces listed there.

In Section \ref{subsec:VgXg}, we discuss the coordinates we use for the representations $V_g$ and the equations of the homogeneous varieties $X_g$ in these coordinates. In Section \ref{subsec:points}, we describe points in the Grassmannians corresponding to the singular curves and surfaces appearing in the Main Theorem, and give a link to the author's website where additional \texttt{Macaulay2} code can be found verifying these claims.  Finally, in Section \ref{subsec:ribbon}, we demonstrate in an example how we obtained these points.

\subsection{Coordinates for $V_g$ and equations of $X_g$}\label{subsec:VgXg}

We discuss the coordinates we use for the representations $V_g$ and
briefly comment on how we obtain  equations of the homogeneous
varieties $X_g$ in these coordinates. Full lists of 
these equations are available at \cite{Code}.  

\subsubsection{Genus 7}
Mukai's coordinates for $V_7$ are $x_0$, $x_{12}$, $x_{13}$,
$x_{14}$, $x_{15}$, $x_{23}$, $x_{24}$, $x_{25}$, $x_{34}$, $x_{35}$, $x_{45}$,
$x_{1234}$, $x_{1235}$, $x_{1245}$, $x_{1345}$, $x_{2345}$. The homogeneous variety $X_7$ is the orthogonal Grassmannian
$\OG$, and Mukai gives equations for it in \cite{Mukai1995}*{(0.1)}.

\subsubsection{Genus 8} We use the Pl\"{u}cker coordinates
for $\mywedge^2 \mathbb{C}^{6}$, and the equations defining $X_8 =
\Gr(2,6)$ are well-known. 

\subsubsection{Genus 9} \label{subsubsec: SpGr36}
First, we describe a basis of the vector space $V_9$.
The representation can be constructed as the kernel of the
  contraction map $\varphi_3: \mywedge^3 V \rightarrow V$
given by 
\[
\varphi_3(v_1 \wedge v_2 \wedge v_3) = \sum_{i < j} Q(v_i,v_j)
  (-1)^{i+j-1} v_k,
\]
where \(V\) is the standard representation and \(Q\) is the symplectic
  form. See \cite{FH}*{Lecture 17}.

Write \(e_{i_{1}i_{2}i_{3}}\) for the wedge product \(e_{i_1} \wedge
e_{i_2} \wedge e_{i_3}\). We chose the following basis for \(V_9\).
\begin{gather*}
\{ e_{123}, e_{234},e_{134}-e_{235},-e_{124}-e_{236},
-e_{135},e_{125}-e_{136}, e_{126},\\ e_{156}, e_{146}-e_{256},
-e_{145}-e_{356}, -e_{246}, e_{245}-e_{346},e_{345},e_{456}\}.
\end{gather*}

Let $z_0,\ldots,z_{13}$ be the dual basis. Mukai describes equations for the symplectic Grassmannian $\SpGr$ in
 \cite{Mukai2010}*{Section 2}. Write
 \begin{align*}
   y &= z_{0} \\
   X &= \left[ \begin{array}{rrr} z_1 & z_2 & z_3 \\ z_2 & z_4 & z_5 \\ z_3 & z_5 & z_6    \end{array} \right] \\
   Y &= \left[ \begin{array}{rrr} z_7 & z_8 & z_9 \\ z_8 & z_{10} & z_{11} \\ z_9 & z_{11} & z_{12}    \end{array} \right] \\
   x &= z_{13} \\
 \end{align*}

 Let $X' = \operatorname{Cof}(X)$, and $Y' = \operatorname{Cof}(Y)$ be the matrices of cofactors of $X$ and $Y$ respectively. Then the equations of $\SpGr$ are given by
 \[
  X' = yY, \qquad Y' = xX, \qquad \text{ and } XY = xy I_3.
 \]

 \subsubsection{Genus 10} We use the basis
 $H_1,H_2,X_1,X_2,X_3,X_4,X_5,X_6,Y_1,Y_2,Y_3,Y_4,Y_5,Y_6$ of the Lie
 algebra $\mathfrak{g}_2$ with brackets defined by Table 22.1 in
 \cite{FH}.  We use the algorithm given by Lichtenstein in
 \cite{Lichtenstein} for computing the minimal closed orbit variety in
 an irreducible representation to obtain quadrics defining the adjoint
 variety of $G_2$; see \cite{Code}.

\subsection{Points in Mukai's parameter space corresponding to the singular curves and surfaces in the Main Theorem}\label{subsec:points}
\begin{theorem} \label{thm: points in parameter space}
  Intersecting the following linear spaces with the homogeneous varieties $X_g$ yield the singular curves and surfaces listed below.

Genus 7:
  \begin{enumerate}
  \item[(a)] The genus 7 tangent developable:
    \[ V(x_{13}, x_{23}+\frac{5}{3} x_{15}, x_{24}, x_{34}-5 x_{25}, x_{45}-\frac{4}{3} x_{35}, x_{1245}-\frac{9}{8} x_{1235}, x_{1345}+\frac{1}{2} x_{12}, x_{2345}+\frac{2}{15} x_{14})
     \]
\item[(b)] The genus 7 cuspidal cubic with 7-gonal symmetry:
    \[ V(x_{13}, x_{23}+\frac{5}{3} x_{15}, x_{24}, x_{34}-5 x_{25},
      x_{45}-\frac{4}{3} x_{35}, x_{1245}-\frac{9}{8} x_{1235},
      x_{1345}+\frac{1}{2} x_{12}, x_{2345}+\frac{2}{15}
      x_{14},x_{1234}-\frac{1}{30} x_0)
     \]
   \item[(c)] The genus 7 balanced K3 carpet:
  \[
V(x_0, x_{12} - x_{1345}, x_{13} - x_{2345}, x_{24} - 2 x_{15}, x_{34} - \frac{1}{2}x_{25}, x_{45}, x_{1234}, x_{1235})
   \]
 \item[(d)] The genus 7 balanced ribbon:
   \[
     V(x_0, x_{12}-x_{1345}, x_{13}-x_{2345}, x_{24}-2 x_{15},
     x_{34} - \frac{1}{2}x_{25}, x_{45}, x_{1234}, x_{1235}, x_{23}-x_{14} )
   \]
 \item[(e)] The genus 7 graph curve for $\Gamma_7$:
   \[ V(x_{2345},x_{1245},x_{1234},x_{35},x_{23}-x_{25},x_{15}+x_{25}-x_{35}-x_{45},x_{14},x_{12}-x_{13},x_0+x_{25}-x_{35}-x_{45})
    \]
\end{enumerate}

Genus 8: 
  \begin{enumerate}
  \item[(a)] The genus 8 tangent developable:
  \[
V(x_{23} - \frac{5}{3} x_{14}, x_{24} - 5 x_{15}, x_{25} - 15 x_{16}, x_{34} - 20 x_{16}, x_{35} - 5 x_{26}, x_{45} - \frac{5}{3}x_{36})
   \]    
\item[(b)] The genus 8 cuspidal cubic with 8-gonal symmetry:
  \[
V(x_{23} - \frac{5}{3} x_{14}, x_{24} - 5 x_{15}, x_{25} - 15 x_{16}, x_{34} - 20 x_{16}, x_{35} - 5 x_{26}, x_{45} - \frac{5}{3}x_{36}, x_{12}-x_{56})
   \]
 \item[(c)] A genus 8 reducible surface with four components, each
   generically nonreduced:
   \[
 V(x_{12}-x_{34}, x_{12}-x_{13}, x_{14}-x_{45}, x_{16}-x_{23}, x_{24}-x_{35}, x_{26}-x_{36})
\]   
\item[(d)] A genus 8 reducible curve with five components, each
  generically nonreduced:
  \[
 V(x_{12}-x_{34}, x_{12}-x_{13}, x_{14}-x_{45}, x_{16}-x_{23}, x_{24}-x_{35}, x_{26}-x_{36}, x_{15}-x_{46} ) 
\]
\item[(e)] The genus 8 graph curve for $\Gamma_8$:
  \[
V(x_{12}-32 x_{46},  x_{13}+(64\sqrt{2}-96) x_{56}, x_{14}+\frac{1}{32}
x_{23},   x_{15}-x_{24}, x_{16}+\frac{1}{32}(2\sqrt{2}+3) x_{35}, x_{25}-x_{34},
x_{26}-x_{45})
  \]

\end{enumerate}

Genus 9:
  \begin{enumerate}
  \item[(a)] The genus 9 tangent developable:
\[
V( z_3-\frac{135}{2}z_4,z_9-\frac{1}{432} z_{10},z_5-45z_{12},z_6 + 36z_{11})      
\]   
\item[(b)] The genus 9 cuspidal cubic with 9-gonal symmetry:
\[
V( z_3-\frac{135}{2}z_4,z_9-\frac{1}{432} z_{10},z_5-45z_{12},z_6 +
36z_{11}, z_0 + \frac{1}{5} z_{13})      
\]     
\item[(c)] The genus 9 balanced K3 carpet:
  \[
V( z_0, z_{13}, z_3+\frac{1}{2}z_4, z_9+\frac{1}{2}z_{10})
\]
\item[(d)] The genus 9 balanced ribbon:
  \[
V(   z_0, z_{13}, z_3+\frac{1}{2}z_4, z_9+\frac{1}{2}z_{10}, z_1-4z_7)
\]  
\item[(e)] The genus 9 graph curve for $\Gamma_9$:
  \[
    V(   z_{0} -(12\sqrt{3}-18) z_{13},  z_4 -
    (\frac{4}{9}\sqrt{3}+\frac{2}{3})z_3, z_6 - (12 \sqrt{3}-21)z_1,
    z_{10} - (-6\sqrt{3}+9)z_9, z_{12}+\frac{1}{3}z_7).  
\]
 \end{enumerate}

Genus 10:
  \begin{enumerate}
  \item[(a)] The genus 10 tangent developable:
  \[
V(H_1-\frac{4}{9}H_2, X_3-\frac{4}{27}Y_6, Y_3-\frac{5}{4}X_6)
  \]    
\item[(b)] The genus 10 cuspidal cubic with 10-gonal symmetry:
  \[
V(H_1-\frac{4}{9}H_2, X_3-\frac{4}{27}Y_6, Y_3-\frac{5}{4}X_6,X_5+\frac{8}{75}Y_5)
  \]
\item[(c)] A genus 10 reducible surface with two components, each generically nonreduced:
  \[
 V(X_1-Y_1,Y_2-Y_3,X_4-X_5)
\]   
\item[(d)] A genus 10 reducible curve with four components, each generically nonreduced:
  \[
 V(X_1-Y_1,Y_2-Y_3,X_4-X_5,H_1-Y_5)
\]
\end{enumerate}
\end{theorem}

\begin{proof}
For the genus 7 curves above (examples 7b, 7d, and 7e), these points in the Grassmannian $\Gr(7,16)$ are given in \cite{Swinarski2025}, Propositions 3.1, 4.1, and 5.1.

For the genus 8 and 10 nonreduced curves and surfaces (examples 8c, 8d,
10c, and 10d), we use \texttt{Macaulay2} to compute primary
decompositions of the ideal of the intersection $P \cap X_g$ and its
radical to establish the description given.

We sketch the approach used for the remaining examples. For each
curve $C \subset \mathbb{P}^g$, we obtain equations of its ideal $I(C) \subseteq k[x_0,\ldots,x_g]$ from the
literature; see the references below. We choose an
isomorphism $P \cong \mathbb{P}^g$, and provide an explicit
isomorphism $\varphi: \mathbb{P}^g \rightarrow \mathbb{P}^g$ inducing
a map that carries $I(P \cap X_g)$ to $I(C)$. (An analogous procedure is used for surfaces,
replacing $\mathbb{P}^g$ by $\mathbb{P}^{g+1}$.)

See the code on the author's website \cite{Code} for precise formulas for these isomorphisms.

Here are some references for equations (or
algorithms for producing equations) of these singular curves and surfaces
\begin{itemize}
\item Equations for the tangent developable can be obtained by
  starting with a parametrization of this surface (see for instance
  \cite{EL}), then eliminating parameters. The equations can also be
  found explicitly in the literature. See \cite{E}*{Section I.A} for
  quadrics defining the ideal, and \cite{BS1990}*{Prop. 2.9} (who in
  turn cite \cite{SchreyerThesis}) for a Gr\"{o}bner basis of the
  ideal with respect to the grevlex term order. This yields 
  equations for Examples 7a, 8a, 9a, and 10a.
\item $g$-cuspidal curves are obtained as hyperplane sections of the
  tangent developable. Choosing the hyperplane $y_0=y_g$ (in the
  notation of \cite{BS1990}*{Prop. 2.9}) gives a $g$-cuspidal curve
  with $g$-gonal symmetry. This yields equations for Examples 7b, 8b,
  9b, and 10b.
\item Equations of K3 carpets $X(a,b)$ are described in
  \cite{ES2019}*{Theorem 3.5}, and also implemented in the
  \texttt{Macaulay2} package \texttt{K3Carpets}. This yields equations
  for Examples 7c and 9c.
\item Ribbons are hyperplane sections of K3 carpets. Equations of the balanced ribbons are given in \cite{DFS2014}*{Thm.~4.4}. This yields equations
  for Examples 7d and 9d.  
\item An algorithm for producing equations of trivalent graph curves
  is described in  \cite{BE1991}*{Prop.~3.1}. We use this to obtain a
  first set of equations for Examples 7e, 8e, and 9e. For examples 8e
  and 9e, we found it useful to subsequently change coordinates to
  diagonalize the action of one of the automorphisms of largest order.
\end{itemize}

\end{proof}

The statement and proof of Theorem \ref{thm: points in parameter
  space} above do not explain or even hint at how the equations defining these
linear spaces were obtained. We used three strategies.

\begin{itemize}
\item In genus 7, Mukai's construction can be implemented as an
  algorithm. See \cite{Swinarski2025} Sections 3 and 4 for details.
\item Many of the examples presented here have rich automorphism
  groups, and we used them to find the linear spaces presented
  here. To illustrate this idea, we present the calculations in detail for
two examples (the genus 9 balanced K3 carpet and balanced ribbon) in Section \ref{subsec:ribbon} below. 
\item We found the generically nonreduced curves and surfaces of genus 8
  and 10 by a computer search. We began by studying linear spaces defined by binomial hyperplanes and found GIT semistable surfaces of genus 8 and 10 with every component generically nonreduced. Then we searched hyperplane sections of these surfaces to find the generically nonreduced curves.
\end{itemize}

\subsection{Examples: Linear spaces yielding the genus 9 balanced K3 carpet and balanced ribbon}\label{subsec:ribbon}

We illustrate the ideas used to find many the linear spaces presented in Theorem \ref{thm: points in parameter
  space} for two examples: the genus 9
balanced K3 carpet and balanced ribbon (examples 9c and 9d).

\begin{proposition}
Let $X(4,4)$ be the genus 9 balanced K3 carpet.  Then  $X(4,4) =
\SpGr \cap V(z_0,  z_{13}, z_3+\frac{1}{2}z_4, z_9+\frac{1}{2}z_{10})$.

The genus 9 balanced ribbon is $\SpGr  \cap V(   z_0, z_{13}, z_3+\frac{1}{2}z_4, z_9+\frac{1}{2}z_{10}, z_1-4z_7)$.  
\end{proposition}
\begin{proof}
We approach this calculation in five steps.
\begin{enumerate}
\item[Step 1.] We obtain generators of the ideals of $\SpGr \subset \mathbb{P}^{13}$ and $X(4,4) \subset \mathbb{P}^{9}.$
\item[Step 2.] We compare the representations of the automorphism groups of the
  balanced K3 carpet and ribbon and $\SpGr$ and make
  an ansatz regarding the form of the desired ring homomorphism. 
\item[Step 3.] We compare the quadrics defining $X(4,4)$ and $\SpGr$ and
  define an explicit affine variety $\AffX$ such that finding the
  missing coefficients in the ring homomorphism is equivalent to finding a
  point on $\AffX$. 
\item[Step 4.] We find a point on the variety $\AffX$. This completes our calculation for $X(4,4)$.
\item[Step 5.] We obtain the balanced ribbon as a hyperplane section of $X(4,4)$.
\end{enumerate}

The \texttt{Macaulay2} code for the calculations below is posted on the author's website \cite{Code}.

\textit{Step 1:} First, we obtain the quadrics defining the ideals of $\SpGr$ and $X(4,4)$. A set of quadrics $\{f_i\}$ defining $\SpGr$ is obtained as described in Section \ref{subsubsec: SpGr36}. A set of quadrics $\{g_j\}$ defining $X(4,4)$ is obtained from \cite{ES2019}*{Theorem 3.5}, which is implemented in the \texttt{Macaulay2} package \texttt{K3Carpets}.

\textit{Step 2.} We briefly review some features of representation
$V_9 = V(0,0,1)$ of $\mathfrak{sp}(6)$; see \cite{FH}*{Section 17.1}
for more details. In the basis given
by $L_1,L_2,L_3$, the weights of $V_9$ lie on the
vertices and midpoints of the faces of a cube in
$\mathbb{R}^3$. See Figure \ref{fig:V9weights}.  

\begin{figure}[h]
        \caption{Weights of $V_9$}
        \label{fig:V9weights}    
\begin{center}
  \scalebox{0.15}{\includegraphics{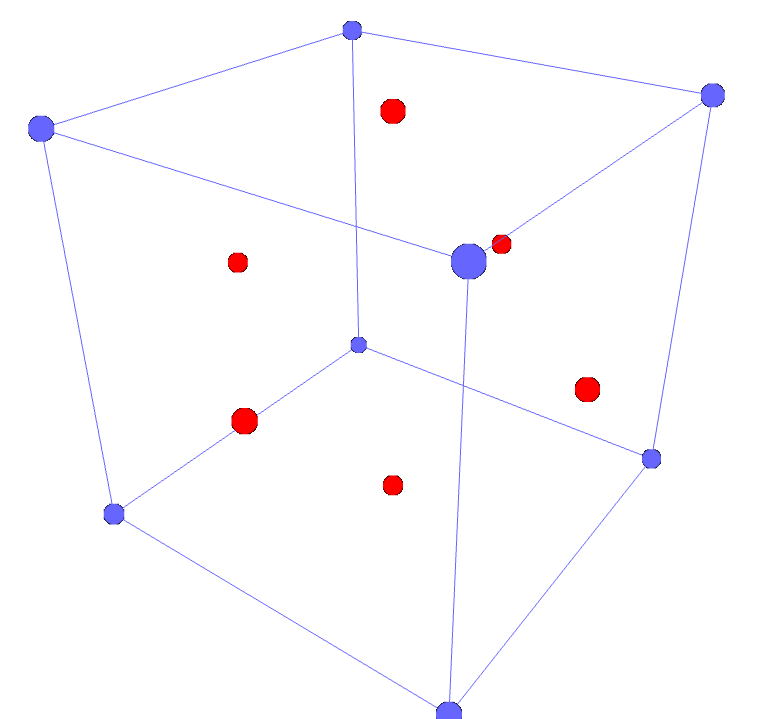}}
\end{center}
\end{figure}
The balanced K3 carpet has automorphism group $\SL_2 \times \SL_2$.
This automorphism group contains a $\mathbb{G}_m$ acting with weights
$-4,-3,-2,-1,0,0,1,2,3,4$. The balanced ribbon is a hyperplane section
of the balanced K3 carpet, and it has a $\mathbb{G}_m$ action with weights
$-4,-3,-2,-1,0,1,2,3,4$.

How could we map the weights of $V_9$ onto these $\mathbb{G}_m$
weights for the K3 carpet? A little geometric tinkering suggests a
solution. Projecting the weights 
onto the line through the origin in the direction of $(3,2,1)$ yields
a set containing nine distinct, evenly spaced points. Specifically, consider the lattice generated by $(0,0,0)$ and $\frac{1}{14}(3,2,1)$,
which is the projection of $L_3$ onto $(3,2,1)$. Then the images of the weights of $V_9$ in this lattice are 
are $-6, -4, -3, -2, -2, -1, 0, 0, 1, 2, 2, 3, 4, 6$. See Figure
\ref{fig:V9weightprojection}.

\begin{figure}[h]
        \caption{Projection of weights of $V_9$}
        \label{fig:V9weightprojection}    
\begin{center}
  \scalebox{0.15}{\includegraphics{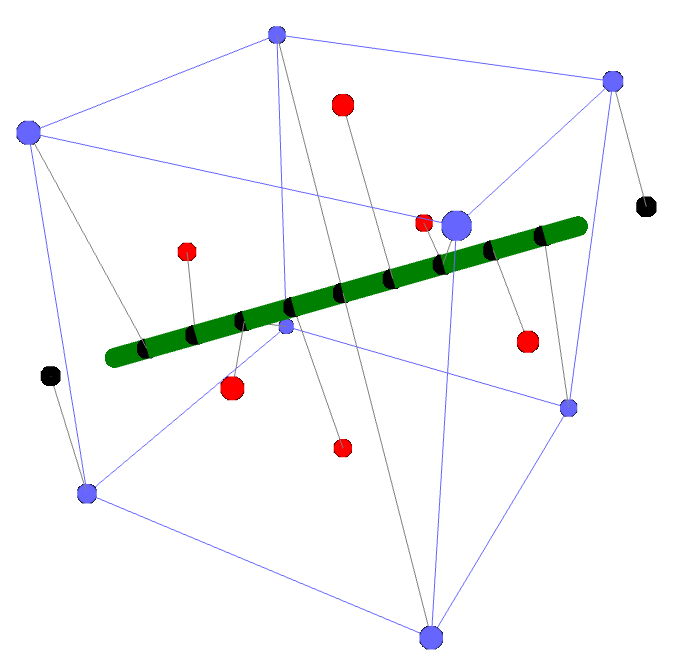}}
\end{center}
\end{figure}

To match the $\mathbb{G}_m$ weights on $X(4,4)$, we must kill the weights that map to $\pm 6$, and
choose one-dimensional subspaces of the weights that map to $\pm 2$, which
have dimension two each.

We write $I \subseteq R = \mathbb{C}[z_0,\ldots,z_{13}]$ for the ideal of $\SpGr$ and
$J \subseteq S=\mathbb{C}[a_0,\ldots,a_9]$ for the ideal of the
balanced K3 carpet. The observations above suggest that we search for a ring homomorphism $\psi:
R\rightarrow S$
of the following form.
\begin{equation}
 \begin{array}{lcrclcr}
z_0 &\mapsto &  0   & \qquad & z_7 &\mapsto & c_7a_4\\
z_1 &\mapsto & c_1a_5 && z_8 &\mapsto &c_8a_3\\
z_2 &\mapsto & c_2 a_6 && z_9 &\mapsto & c_9a_2\\
z_3 &\mapsto & c_3 a_7 && z_{10} &\mapsto &c_{10}a_2\\
z_4 &\mapsto & c_4a_7 && z_{11} &\mapsto & c_{11}a_1\\
z_5 &\mapsto & c_5a_8 && z_{12} &\mapsto & c_{12} a_0\\
z_6 &\mapsto & c_6 a_9 && z_{13} &\mapsto & 0
 \end{array}
\end{equation}

This in turn yields hyperplanes of the form $z_0 = 0$, $z_{13}
= 0$, $c_4 z_3-c_3z_4$, and $c_{10} z_9-c_{9} z_{10}$.

\textit{Step 3:} We compare the quadrics defining the ideals of $\SpGr$ and $X(4,4)$. 

We find that for 19 of the 21 quadrics $f_i$, the support of $\psi(f_i)$ matches the support of one of the quadrics $g_j$. The remaining two pairs of quadrics can also be matched after replacing the quadrics $g_j$ by a linear combination in a straightforward way (e.g. replace $a_{2}^{2}-a_{1}a_{3}$ by $(a_{2}^{2}-a_{1}a_{3}) + (a_{1}a_{3}-a_{0}a_{4})$ to obtain $a_{2}^{2}-a_{0}a_{4}$). 

We thus obtain the following table.
\[
\begin{array}{lll}
\text{Quadrics $f_i$ defining $\SpGr$} & \psi(f_i) & \text{Quadrics $g_i$ defining $X(4,4)$}\\
-z_5^2+z_4 z_6-z_0 z_7    &    -c_5^2 a_8^2+c_4 c_6 a_7 a_9    &    a_8^2-a_7 a_9\\
z_3 z_5-z_2 z_6-z_0 z_8    &    c_3 c_5 a_7 a_8-c_2 c_6 a_6 a_9    &    a_7 a_8-a_6 a_9\\
-z_3 z_4+z_2 z_5-z_0 z_9    &    -c_3 c_4 a_7^2+c_2 c_5 a_6 a_8    &    a_7^2-a_6 a_8\\
-z_3^2+z_1 z_6-z_0 z_10    &    -c_3^2 a_7^2+c_1 c_6 a_5 a_9    &    a_7^2-a_5 a_9\\
z_2 z_3-z_1 z_5-z_0 z_11    &    c_2 c_3 a_6 a_7-c_1 c_5 a_5 a_8    &    a_6 a_7-a_5 a_8\\
-z_2^2+z_1 z_4-z_0 z_12    &    -c_2^2 a_6^2+c_1 c_4 a_5 a_7    &    a_6^2-a_5 a_7\\
-z_11^2+z_10 z_12-z_1 z_13    &    -c_{11}^2 a_1^2+c_{10} c_{12} a_0 a_2    &    a_1^2-a_0 a_2\\
z_9 z_11-z_8 z_12-z_2 z_13    &    c_9 c_{11} a_1 a_2-c_8 c_{12} a_0 a_3    &    a_1 a_2-a_0 a_3\\
-z_9 z_10+z_8 z_11-z_3 z_13    &    -c_9 c_{10} a_2^2+c_8 c_{11} a_1 a_3    &    a_2^2-a_1 a_3\\
-z_9^2+z_7 z_12-z_4 z_13    &    -c_9^2 a_2^2+c_7 c_{12} a_0 a_4    &    a_2^2-a_0 a_4\\
z_8 z_9-z_7 z_11-z_5 z_13    &    c_8 c_9 a_2 a_3-c_7 c_{11} a_1 a_4    &    a_2 a_3-a_1 a_4\\
-z_8^2+z_7 z_10-z_6 z_13    &    -c_8^2 a_3^2+c_7 c_{10} a_2 a_4    &    a_3^2-a_2 a_4\\
z_1 z_7+z_2 z_8+z_3 z_9-z_0 z_13    &    c_1 c_7 a_4 a_5+c_2 c_8 a_3 a_6+c_3 c_9 a_2 a_7    &    a_4 a_5-2 a_3 a_6+a_2 a_7\\
z_1 z_8+z_2 z_10+z_3 z_11    &    c_1 c_8 a_3 a_5+c_2 c_{10} a_2 a_6+c_3 c_{11} a_1 a_7    &    a_3 a_5-2 a_2 a_6+a_1 a_7\\
z_1 z_9+z_2 z_11+z_3 z_12    &    c_1 c_9 a_2 a_5+c_2 c_{11} a_1 a_6+c_3 c_{12} a_0 a_7    &    a_2 a_5-2 a_1 a_6+a_0 a_7\\
z_2 z_7+z_4 z_8+z_5 z_9    &    c_2 c_7 a_4 a_6+c_4 c_8 a_3 a_7+c_5 c_9 a_2 a_8    &    a_4 a_6-2 a_3 a_7+a_2 a_8\\
z_2 z_8+z_4 z_10+z_5 z_11-z_0 z_13    &    c_2 c_8 a_3 a_6+c_4 c_{10} a_2 a_7+c_5 c_{11} a_1 a_8    &    a_3 a_6-2 a_2 a_7+a_1 a_8\\
z_2 z_9+z_4 z_11+z_5 z_12    &    c_2 c_9 a_2 a_6+c_4 c_{11} a_1 a_7+c_5 c_{12} a_0 a_8    &    a_2 a_6-2 a_1 a_7+a_0 a_8\\
z_3 z_7+z_5 z_8+z_6 z_9    &    c_3 c_7 a_4 a_7+c_5 c_8 a_3 a_8+c_6 c_9 a_2 a_9    &    a_4 a_7-2 a_3 a_8+a_2 a_9\\
z_3 z_8+z_5 z_10+z_6 z_11    &    c_3 c_8 a_3 a_7+c_5 c_{10} a_2 a_8+c_6 c_{11} a_1 a_9    &    a_3 a_7-2 a_2 a_8+a_1 a_9\\
z_3 z_9+z_5 z_11+z_6 z_12-z_0 z_13    &    c_3 c_9 a_2 a_7+c_5 c_{11} a_1 a_8+c_6 c_{12} a_0 a_9    &    a_2 a_7-2 a_1 a_8+a_0 a_9
\end{array}
\]

We now define an affine variety $\AffX$ as follows. First, we require that
$\psi(f_i)$ is a scalar multiple of $g_i$ for each $i$. For
example, when $i=1$, we have $\psi(f_1) =-c_5^2 a_8^2+c_4 c_6 a_7 a_9 $ and $g_1 = a_8^2-a_7 a_9$. Then 
$\psi(f_i)$ is a scalar multiple of $g_i$ if the $2 \times 2$
minors of the matrix
\[
 \left[ \begin{array}{cc}
    -c_5^2 & c_4 c_6 \\
    1 & -1
  \end{array} \right]
\]
vanish. Thus, we include the equation $c_5^2-c_4c_6$ in the defining ideal of $\AffX$ for
$i=1$. We treat the remaining quadrics in a similar fashion. 

Finally, we also require that the coefficients $c_j \neq 0$ for
$j=1,\ldots, 12$. To ensure this, we include equations of the form
$c_j d_j -1$ for each $j$ in this range. 

This yields the following ideal of an affine variety $\AffX \subset \mathbb{A}^{24}$.

\begin{gather*}
  \AffX = \langle c_5^2-c_4 c_6, -c_3 c_5+c_2 c_6, c_3 c_4-c_2 c_5, c_3^2-c_1 c_6, -c_2 c_3+c_1 c_5, c_2^2-c_1 c_4,\\
c_{11}^2-c_{10} c_{12}, -c_9 c_{11}+c_8 c_{12}, c_9 c_{10}-c_8 c_{11}, c_9^2-c_7 c_{12}, -c_8 c_9+c_7 c_{11}, c_8^2-c_7 c_{10},\\
-2 c_1 c_7-c_2 c_8, c_1 c_7-c_3 c_9, -2 c_1 c_8-c_2 c_{10}, c_1 c_8-c_3 c_{11}, -2 c_1 c_9-c_2 c_{11}, c_1 c_9-c_3 c_{12},\\
-2 c_2 c_7-c_4 c_8, c_2 c_7-c_5 c_9, -2 c_2 c_8-c_4 c_{10}, c_2 c_8-c_5 c_{11}, -2 c_2 c_9-c_4 c_{11}, c_2 c_9-c_5 c_{12},\\
-2 c_3 c_7-c_5 c_8, c_3 c_7-c_6 c_9, -2 c_3 c_8-c_5 c_{10}, c_3 c_8-c_6 c_{11}, -2 c_3 c_9-c_5 c_{11},  \\
c_3 c_9-c_6 c_{12},c_1 d_1-1,c_2 d_2-1,c_3 d_3-1,c_4 d_4-1,c_5 d_5-1,c_6 d_6-1,\\
c_7 d_7 -1, c_8 d_8 -1, c_9 d_9 -1, c_{10} d_{10} -1, c_{11} d_{11} -1, c_{12} d_{12} -1\rangle.
\end{gather*}

We find that $\AffX$ has dimension 4, and is reduced and irreducible.

\textit{Step 4.} We find a point on the variety $\AffX$. By intersecting $\AffX$ with the hyperplanes $c_1 = 1$, $c_2 = 1$, $c_{11}=1$, $c_{12}=1$, we obtain a point on $\AffX$ whose first twelve coordinates $c_i$ are   $(1, 1, -\frac{1}{2}, 1, -\frac{1}{2}, \frac{1}{4}, \frac{1}{4}, -\frac{1}{2}, -\frac{1}{2}, 1, 1, 1)$. 

This yields a set of coefficients $c_i$ defining a ring homomorphism $\psi$ with
the desired properties.

In summary, let $\psi: \mathbb{C}[z_0,\ldots,z_{13}] \rightarrow \mathbb{C}[a_0,\ldots,a_9] $ be the following map.
 \begin{equation}
 \begin{array}{lcrclcr}
z_0 &\mapsto &  0   & \qquad & z_7 &\mapsto &\frac{1}{4}a_4\\
z_1 &\mapsto & a_5 && z_8 &\mapsto  &-\frac{1}{2} a_3\\
z_2 &\mapsto & a_6 && z_9 &\mapsto  & -\frac{1}{2} a_2\\
z_3 &\mapsto & -\frac{1}{2} a_7 && z_{10} &\mapsto &a_2\\
z_4 &\mapsto & a_7 && z_{11} &\mapsto & a_1\\
z_5 &\mapsto & -\frac{1}{2} a_8 && z_{12} &\mapsto & a_0\\
z_6 &\mapsto & \frac{1}{4} a_9 && z_{13} &\mapsto & 0
 \end{array}
\end{equation}
 Then a calculation in \texttt{Macaulay2} confirms that $\psi(I(\SpGr)) = I(X(4,4))$.

 Thus $X(4,4) = \SpGr \cap H_1 \cap H_2 \cap H_3 \cap H_4$, where the hyperplanes $H_1,\ldots,H_4$ are defined by
 \begin{gather*}
   z_0,\\
   z_{13},\\
   z_3+\frac{1}{2}z_4,\\
  z_9+\frac{1}{2}z_{10}.
 \end{gather*}
 respectively.

 \textit{Step 5.} Finally, we obtain the balanced ribbon as a hyperplane section of the balanced K3 carpet. In the coordinates used here, the additional hyperplane is $a_4=a_5$, which corresponds to the hyperplane $z_1-4z_7$.
\end{proof}

\section{Proof of the Main Theorem and remarks on computational performance} \label{sec: proof}

\subsection{Strategy} Consider the descriptions of the polynomials
$F_{g,d}$ given in Section \ref{sec: invariant polynomials}.

These polynomials can be constructed by repeating the following task
several times: \textit{given a $\mathfrak{g}$-module $W$ and a weight $\mu$,
find a submodule $U\subset W$
that is isomorphic to the irreducible representation $V(\mu)$.}  

We can do this as follows.
\begin{enumerate}
\item Compute the explicit action of the lowering operators of
  $\mathfrak{g}$ on $W$.
\item Find a highest weight vector of weight $\mu$ in $W$.
\item Choose a basis of $V(\mu)$ and express each basis element
  as a word in the lowering operators applied to a highest weight vector.
\item Evaluate these words acting on the highest weight vector to obtain
  a basis of $U \cong V(\mu)$.
\end{enumerate}

Each of these steps is implemented in the \texttt{LieAlgebraRepresentations} package
for \texttt{Macaulay2}.

\subsection{Computational performance}
Here are some statistics associated to the calculations associated to
constructing the polynomials $F_{g,d}$. These were performed on a 2020
MacBook Pro with Apple M1 chip and 16GB memory, running macOS Sonoma
14.7 and \texttt{Macaulay2} version 1.25.11. 
\begin{center}
\begin{tabular}{lll}
  Polynomial & Time to construct $F_{g,d}$ (hh:mm:ss) & Output file size \\
  $F_{7,1}$ & 24:02:24 & 12.6MB \\
  $F_{7,2}$ & 00:50:12 & 123KB \\
  $F_{8,1}$ & 04:33:03 & 4.1MB \\
  $F_{8,2}$ & 08:40:52 & 602KB \\
  $F_{9,1}$ & 00:01:13 & 83KB \\
  $F_{9,2}$ & 00:00:27 & $<1$KB \\
  $F_{10,1}$ & 00:00:12& 2KB \\
  $F_{10,2}$ & 00:00:09 & $<1$KB \\
\end{tabular}
\end{center}

We see that in terms of time and memory, the calculations to construct $F_{g,d}$ are easier for surfaces than they are for curves, and easier as $g$ increases. This
is opposite the usual trends in algebraic geometry and can be partly explained
by three observations.

First, as shown in the table below, the dimensions of the groups and representations involved
decrease as $g$ increases. Second, for $g \geq 8$, we have $g \geq \frac{1}{2} \dim V_g^{*}$, so
that $\dim \mywedge^{g+1} V_g^{*} < \dim \mywedge^{g} V_g^{*} $. Finally, the representations $\mywedge^{k} V_g^{*} $ (where $k=g$ for curves and $k=g+1$ for surfaces)  break up into larger number of smaller irreducibles as $g$ increases, and for
surfaces compared to curves. 

\begin{center}
\begin{tabular}{llll}
  Moduli problem & $\dim V_g^{*}$ & $\dim G_g$ & Dimensions of irreducibles in $\mywedge^{k} V_g^{*} $\\
 $\mukai^{1}_7$ & 16 & 45 & 8800, 2640 \\
  $\mukai^{2}_7$ & 16 & 45 & 8085, 4125, 660 \\
 $\mukai^{1}_8$ & 15 & 35 & 2430, 2205, 1800 \\
  $\mukai^{2}_8$ &15 & 35 & 3675, 840, 490 \\
 $\mukai^{1}_9$ &14 & 21&  1386, 350, 252, 14\\
  $\mukai^{2}_9$ & 14 & 21 & 525, 385, 90, 1 \\
 $\mukai^{1}_{10}$ & 14 & 14 & 448, 189, 182, 77, 64, 27, 14 \\
 $\mukai^{2}_{10}$ & 14 & 14 & 182, 77, 77, 27, 1 \\           
\end{tabular}
\end{center}

\subsection{Proof of the main theorem}

Now that we have the polynomials $F_{g,d}$ from Section \ref{sec:
  invariant polynomials} and the linear spaces described in Theorem
\ref{thm: points in parameter space}, we can prove the Main Theorem. 

\begin{proof}
  For each singular curve or surface described in Theorem
  \ref{thm:main}, we evaluate the corresponding $G_g$-invariant
  polynomial described in Section \ref{sec: invariant polynomials} at
  the point in the parameter space $\Gr(g,V_g)$ or $\Gr(g+1,V_g)$
  described in Theorem \ref{thm: points in parameter space}, and
  obtain a nonzero value. See the author's website \cite{Code} for the
  \texttt{Macaulay2} code.
\end{proof}


\subsection{Invariant polynomials for $\mukai^{2}_8$}
We studied some additional invariant polynomials related to Mukai's
model $\mukai^{2}_8$. We describe them up to scaling. 

First, we define several $\SL_6$ modules as follows. We have $\mywedge^9 V(0,1,0,0,0)^{*} \cong W_1 \oplus W_2 \oplus W_3$, where
$W_1 \cong V(1,1,0,1,1)$, $W_2 \cong V(3,0,1,0,0)$, and $W_3 \cong
V(0,0,0,3,0)$. These submodules have dimensions
3675, 840, and 490, respectively. Finally, $\Sym^2 W_2$ contains a
summand $U \cong V(0,3,0,0,0)$.

Now we define four invariant polynomials as follows.

\begin{enumerate}
\item[(a)] $\Sym^3W_3$ contains an invariant that we call $F_{8,2a}$. We
  say $F_{8,2a}$ is supported on $W_3$.
\item[(b)] $\Sym^3  U$ contains an invariant. This defines
an invariant that we call $F_{8,2b}$. It is supported on $W_2$.  
\item[(c)] $W_3 \otimes U$ contains an invariant. This defines
an invariant that we call $F_{8,2c}$. It is supported on $W_2$ and $W_3$
\item[(d)] Finally, $W_1$ is self-dual, so $W_1 \otimes
W_1$ contains an invariant. This defines $F_{8,2d} = F_{8,2}$ in Section
\ref{subsec: genus 8 dim 2}. The polynomial $F_{8,2}$ is
supported on $W_1$.
\end{enumerate}

All three polynomials $F_{8,2a}$, $F_{8,2b}$, and $F_{8,2c}$ are
nonzero on the genus 8 tangent developable (Example 8a in Theorem \ref{thm: points in parameter space}), but vanish on the genus 8
nonreduced surface (Example 8c in Theorem \ref{thm: points in
  parameter space}).

In this sense, determining the GIT semistability of Example 8c was more difficult
than it was for the other examples in this study. It is GIT strictly semistable, and the trivial character is a vertex of the state polytope.

In principle, the general purpose functions in the
\texttt{LieAlgebraRepresentations} package could be used to compute $F_{8,2}$, but our attempts to do so
failed, even when performing the calculations on a server with large
amounts of memory. In any case, a far more reasonable approach is to
use the combinatorial formula described in the Appendix. This
application illustrates the usefulness of this formula.

\subsection{Comparison to earlier work}

In \cite{Swinarski2025}, a less efficient algorithm was used to evaluate the less strategically chosen $\Spin(10)$-invariant polynomial described there, and the calculations analogous to those described  in the proof of  
Theorem \ref{thm:main} required very large amounts of time and memory. They
were accomplished by parallel calculations on four AWS
\texttt{r5.24xlarge} instances, each with 96 vCPUs and 768 GB
memory. This took approximately 36 hours, and cost over \$1000.

In contrast, using the new approach described here, the calculations
to construct the invariant polynomial $F_{7,1}$ run in 24 
hours on a 2020 MacBook Pro. 

Due to the amount of memory required by the approach used in
\cite{Swinarski2025}, the output of the intermediate steps of the
calculation was not saved;
as a result, evaluating new examples with that code would be as expensive as the
first calculation. In contrast, in the approach taken here, all the
data needed to evaluate the polynomials $F_{g,d}$ is saved when the
polynomials are constructed, and the
marginal cost of evaluating these polynomials on new examples is
trivial.

\subsection{Checking invariance}

As mentioned in the previous paragraph, all the
data needed to evaluate the polynomials $F_{g,d}$ is saved when the
polynomials are constructed, and the
marginal cost of evaluating these polynomials on new examples is
trivial compared to the cost of constructing them. In particular, we can experimentally test the $G$-invariance of
the polynomials $F_{g,d}$ by applying a random element of $G$ and
checking that the value of $F_{g,d}$ remains unchanged. See
\cite{Code} for some calculations of this type.

\appendix

\section{A combinatorial formula for an $\SL_n$-invariant}

Let $G$ be a simple algebraic group over $\mathbb{C}$. Let $V(\lambda)$ be an
irreducible representation with highest weight $\lambda$. Then the
space of invariants of the following form is one-dimensional.
\[
\dim (V(\lambda) \otimes V(\lambda^{*}))^{G} = 1
\]

Many bases of the irreducible representations $V(\lambda)$ have been
introduced in the literature, particularly when $G = \SL_n$.  In this
note we work with the Gelfand-Tsetlin basis for these representations. See
\cite{Molev2006}*{Section 2} for an introduction.

Our goal is to prove the following formula. The notation will be
explained below in Section \ref{sec:Notation}.

\begin{theorem} \label{theorem: invariant polynomial formula}
  Let
  \[
F = \sum_{P \in \GT(\lambda)} 
\frac{(-1)^{\ell(P)}}{\| P \| \| P^{*}\|} e_P \otimes e_{P^{*}}.
  \]
Then $F$ is an $\SL_n$-invariant vector in $V(\lambda) \otimes V(\lambda^{*})$.
\end{theorem}

Algebraic combinatorics and invariant theory are research fields with
a long history and a large literature, and it would not surprise the
author if this formula is already known. Please contact
the author with relevant references.

Theorem \ref{theorem: invariant polynomial formula} follows easily
from formulas for the action of $\mathfrak{sl}_n$ in such works as
\cite{Molev2006} or \cite{Z}. Unfortunately, we have not yet located a single
reference that has all the results we need, and it is not clear that
the various sources we draw on adopt the same sign conventions. Thus, we aim to give a
self-contained exposition that explains what we implemented in the
\texttt{LieAlgebraRepresentations} package in \texttt{Macaulay2}.

The following discussion will be mostly
combinatorial in nature and develop only what we need. We adopted this
approach in the hope that it would be efficient and amenable to
formalization in the future. However, this 
leaves out lots of beautiful mathematics. We warmly encourage interested
readers to explore the references cited here for more details.

We computed a number of examples in \texttt{Macaulay2} to
experimentally verify the formulas in this note, and wrote code to verify to illustrate
each of the formulas in this note. These files may be found at the author's
webpage.  

\subsection{Notation} \label{sec:Notation}

\subsubsection{From weights to partitions}

We mostly follow the definitions given by Molev in \cite{Molev2006}*{Section 2},
except that we use $\lambda$ to denote a weight of $\mathfrak{sl}_n$
and denote the associated partition by $\pi(\lambda)$. (In \cite{Molev2006},
$\lambda$ is the partition.)

Let $\lambda = \sum a_i \omega_i$ be a dominant integral weight for
$\mathfrak{sl}_n$, written in the basis of fundamental
dominant weights $\omega_i$. The partition $\pi(\lambda)$ associated
to $\lambda$ has parts
\[
  \pi_i = \sum_{j=i}^{n-1} a_i
\]
for $i=1,\ldots,n-1$, and we set $\pi_n=0$. 

\subsubsection{The pattern $P$}
A \emph{Gelfand-Tsetlin pattern of shape  $\pi(\lambda)$} is a triangular array of the following form.

\[\begin{array}{ccccccccc} x_{n,1} & & x_{n,2} & & x_{n,3} & & \cdots
    && x_{n,n} \\ &x_{n-1,1} & & x_{n-1,2} & & \cdots & x_{n-1,n-1} &
    \\ && \ddots \\ && & x_{2,1} && x_{2,2} \\ &&&& x_{1,1}\end{array}
\]

Each entry $x_{i,j}$ is a nonnegative integer, the top row $x_{n,i}$ corresponds to $\pi(\lambda)$, and the entries satisfy the inequalities $x_{k,i} \geq x_{k-1,i} \geq x_{k,i+1}$. 

We write $\GT(\lambda)$ for the set of all Gelfand-Tsetlin patterns of
shape $\pi(\lambda)$. This set indexes a basis of $V(\lambda)$.  We write
$e_P$ for the basis element corresponding to the pattern $P \in \GT(\lambda)$.

\subsubsection{The pattern $P^{*}$}
Let $P \in \GT(\lambda)$. We define the \emph{dual pattern} $P^{*}$  by the formula
\[
y_{i,j} = x_{n,1} - x_{j+1-i,j}.
\]

\subsubsection{The numerator $(-1)^{\ell(P)}$}

The \emph{level} $\ell(P)$ is defined as the level of the
weight $\wt(P)$. We briefly recall these definitions.

Let  $P \in \GT(\lambda)$. Then we define the \emph{content} of $P$ as
$(c_1,\ldots,c_n)$, where
\[
c_i = \sum_{j=1}^{i+1}x_{i+1,j} - \sum_{j=1}^{i}x_{i,j}.
\]

The \emph{weight} of $P$ is 
\begin{align*}
  \wt(P) &= \sum_{i=1}^{n-1} (c_i - c_{i+1}) \omega_i \\
  &= \sum_{i=1}^{n-1} \left(-\sum_{j=1}^{i-1} x_{i-1,j} + 2\sum_{j=1}^{i} x_{i,j} - \sum_{j=1}^{i+1} x_{i+1,j}\right) \omega_i.
\end{align*}

Let $\mu$ be a weight occuring in the representation $V(\lambda)$. Then we may write $\mu = \lambda - \sum_{i=1}^{n-1} k_i \alpha_i$, where the $\alpha_i$ are the simple roots for $\mathfrak{sl}_n$. Then we define the \emph{level} of $\mu$ as $\ell(\mu) = \sum_{i=1}^{n-1} k_i$. See for instance \cite{deGraaf2001}*{Section 4}.

\subsubsection{The norm $\| P\|$}

Write $l_{i,j} = x_{i,j} -j+1$.

Following \cite{Molev2006}*{Proposition 2.4}, we define $\| P \|$ by 
\[
\| P \|^2 = \displaystyle \prod_{t=2}^{n} \displaystyle \prod_{1 \leq
  p \leq q < t} \frac{(l_{t,p} - l_{t-1,q})!}{(l_{t-1,p} -
  l_{t-1,q})!} \displaystyle \prod_{1 \leq p < q \leq t} \frac{(l_{t,p} - l_{t,q}-1)!}{(l_{t-1,p} - l_{t,q}-1)!} 
\]

We will frequently refer to the four products appearing in this formula as $N_1$, $D_1$, $N_2$, and $D_2$. That is, 
\begin{align*}
  N_1 &= \displaystyle\prod_{t=2}^{n}\prod_{1 \leq p \leq q < t} (l_{t,p} - l_{t-1,q})! \\
  D_1 &= \displaystyle\prod_{t=2}^{n} \prod_{1 \leq p \leq q < t} (l_{t-1,p} - l_{t-1,q})! \\
  N_2 &= \displaystyle\prod_{t=2}^{n} \prod_{1 \leq p < q \leq t} (l_{t,p} - l_{t,q}-1)! \\
  D_2 &= \displaystyle\prod_{t=2}^{n} \prod_{1 \leq p < q \leq t} (l_{t-1,p} - l_{t,q}-1)!
\end{align*}

\subsection{Proof of Theorem \ref{theorem: invariant polynomial formula} }

First, we state an easy lemma for obtaining invariant polynomials from
dual bases. Once again, the author welcomes references.

Let $\lambda$ be a dominant integral weight, and let $\lambda^{*}$ be its dual. Let $V(\lambda)$ and $V(\lambda^{*})$ be $\mathfrak{g}$-modules with these highest weights, and let $\rho_{\lambda}: \mathfrak{g} \rightarrow \mathfrak{gl}(V(\lambda))$ and $\rho_{\lambda^{*}}: \mathfrak{g} \rightarrow \mathfrak{gl}(V(\lambda^{*}))$ be the associated Lie algebra representations.

\begin{definition} \label{def: support negative transpose actions}
We say that bases $\mathscr{B}$ and $\mathscr{B}^{*}$ of $V(\lambda)$ and
$V(\lambda^{*})$ \emph{support negative transpose actions} if for any $X \in \mathfrak{g}$,
the matrices
\begin{align*}
  A &= [\rho_{\lambda} (X)]_{\mathscr{B}\leftarrow
        \mathscr{B}}\\
  B &= [\rho_{\lambda^{*}} (X)]_{\mathscr{B}^{*}\leftarrow
        \mathscr{B}^{*}}
\end{align*}
satisfy $B = -A^t$.
\end{definition}

The point of this definition is that we want to say $\mathscr{B}$ and $\mathscr{B}^{*}$
are dual bases, but we haven't given an explicit isomorphism
$V(\lambda^{*}) \cong (V(\lambda))^{*}$ yet. So instead we find bases that have the property that dual bases ought to have, and use them to define the isomorphism $V(\lambda^{*}) \cong (V(\lambda))^{*}$.

\begin{lemma} \label{lem: invariants from negative transpose actions}
Suppose that $\mathscr{B}$ and $\mathscr{B}^{*}$ support
negative transpose actions. Then  
\[
F = \sum_{i=1}^{N} e_i \otimes e_i^{*}
\]
is a $G$-invariant vector.
\end{lemma}
\begin{proof}
  Let $X \in \mathfrak{g}$. Let $A$ and $B$ be the matrices shown in
  Definition \ref{def: support negative transpose actions}. Then we may compute as follows.
\begin{align*}
X(F) &= \sum_{i=1}^{N} X(e_i \otimes e_{i}^{*}) \\
&= \sum_{i=1}^{N}  X(e_i) \otimes e_{i}^{*} +  \sum_{j=1}^{N}  e_j \otimes X(e_{j}^{*})\\
&= \sum_{i=1}^{N}  (\sum_{k=1}^{N} a_{k,i} e_k) \otimes e_{i}^{*} +\sum_{j=1}^{N}e_j\otimes(\sum_{l=1}^{N}b_{l,j} e_l^{*}) \\
&= \sum_{i=1}^{N}  \sum_{k=1}^{N}  a_{k,i} e_k \otimes e_{i}^{*} +\sum_{j=1}^{N}\sum_{l=1}^{N} b_{l,j} e_j\otimes e_l^{*} \\
&= \sum_{i=1}^{N}  \sum_{k=1}^{N}  a_{k,i} e_k \otimes e_{i}^{*} +\sum_{j=1}^{N}\sum_{l=1}^{N} (-a_{j,l}) e_j\otimes e_l^{*} \\      
&= \sum_{i=1}^{N}  \sum_{k=1}^{N}  a_{k,i} e_k \otimes e_{i}^{*} +\sum_{l=1}^{N}\sum_{j=1}^{N} (-a_{j,l}) e_j\otimes e_l^{*} \\
&= 0.   
 \end{align*}

This implies $F$ is a highest weight vector of weight 0, hence an invariant.
\end{proof}

To prove Theorem \ref{theorem: invariant polynomial formula}, we scale
the Gelfand-Tsetlin bases of $V(\lambda)$ and $V(\lambda^{*})$ and
introduce signs to obtain bases supporting negative transpose actions, and
then apply Lemma \ref{lem: invariants from negative transpose actions}.

%

%
%
%

\begin{proposition} \label{prop: negative transpose actions}

The bases $\mathscr{B} = \{ \frac{1}{\|P \|} e_{P}\} $ and
$\mathscr{B}^{*} = \{ \frac{(-1)^{\ell(P^{*})}}{\|  P^{*} \|}
e_{P^{*}} \}$ support
negative transpose actions.

\end{proposition}

To prove Proposition \ref{prop: negative transpose actions}, we use an
explicit formula for the action of the operators $X_{\alpha_k}$ and $Y_{\alpha_k}$ on
the bases $\mathscr{B}$ and $\mathscr{B}_{+}^{*} = \{ \frac{1}{\|  P^{*} \|}
e_{P^{*}} \}$. Formulas of this
type are found in \cite{GT}*{Eqn. (4) and (6)} and \cite{Z}*{\S 70 Thm. 7}, but
we develop this beginning with Molev's formulas to maintain consistent 
sign conventions across all formulas. 

\begin{proposition}[\cite{Molev2006}*{Theorem 2.3}] \label{prop: Molev
    action} \mbox{}
  \begin{enumerate}
    \item The coefficient of $e_{P + \delta_{k,i}}$ in $X_{\alpha_k} e_P$ is 
  \[
    \frac{(-1)
      \displaystyle \prod_{q=1}^{k+1}
      (l_{k,i}-l_{k+1,q})}{\displaystyle \prod_{\substack{q=1,\\q\neq i}}^{k} (l_{k,i}-l_{k,q})}.
    \]
  \item The coefficient of $e_{P - \delta_{k,i}}$ in $Y_{\alpha_k} e_P$ is 
      \[
    \frac{
      \displaystyle \prod_{q=1}^{k-1}
      (l_{k,i}-l_{k-1,q})}{\displaystyle \prod_{\substack{q=1,\\q\neq i}}^{k} (l_{k,i}-l_{k,q})}.
    \]
    \end{enumerate}
\end{proposition}

To get the action on the normalized bases, we first obtain a formula
for the following ratio.
\begin{proposition} \label{prop: ratio of lengths} \mbox{}
  \begin{enumerate}
    \item
  \[
    \frac{\|P + \delta_{k,i} \|}{\| P \|} = \sqrt{\frac{(-1)
    \displaystyle  \prod_{q=1}^{k-1}(l_{k,i}-l_{k-1,q}+1) \displaystyle\prod_{\substack{q=1,\\q\neq i}}^{k} (l_{k,i}-l_{k,q})}{\displaystyle\prod_{q=1}^{k}(l_{k,i}-l_{k,q}+1) \displaystyle\prod_{q=1}^{k} (l_{k,i}-l_{k+1,q})}}.
\]
\item
    \[
    \frac{\|P - \delta_{k,i} \|}{\| P \|} = \sqrt{\frac{
    \displaystyle  \prod_{q=1}^{k+1}(l_{k,i}-l_{k+1,q}-1) \displaystyle\prod_{\substack{q=1,\\q\neq i}}^{k} (l_{k,i}-l_{k,q})}{\displaystyle\prod_{q=1}^{k-1}(l_{k,i}-l_{k-1,q}) \displaystyle\prod_{q=1}^{k} (l_{k,i}-l_{k,q}-1)}}.
\]
\end{enumerate}
\end{proposition}
\begin{proof}
  We prove Part 1. The proof of Part 2 is similar.

  Write $\hat{P}=P+\delta_{k,i}$. This hat notation will only be used
  in this proof.  
  
$P$ and $\hat{P}$ are the same except in the entry $x_{k,i}$,
so any factors of $\| P\|$ and $\| \hat{P} \|$
that do not contain $l_{k,i}$ will cancel.

Write $\|\hat{P}\|^2 = \frac{\hat{N}_1 \hat{N}_2}{\hat{D}_1 \hat{D}_2}$ and $\|P \|^2 = \frac{N_1 N_2}{D_1 D_2}$. We study the ratios $\frac{\hat{N}_1}{N_1}$, $\frac{\hat{D}_1}{D_1} $, etc. one by one.

Consider
\[
\frac{\hat{N_1}}{N_1} = \prod_{t=2}^{n} \prod_{1 \leq p \leq q < t} \frac{(\hat{l}_{t,p}-\hat{l}_{t-1,q})!}{(l_{t,p}-l_{t-1,q})!}.
\]
We have $\hat{l}_{a,b} = l_{a,b}$ unless $a=k$ and $b=i$. 

When $t=k$ and $p=i$, we obtain the following noncancelling factors.
\[
  \prod_{i \leq q < k} \frac{(l_{k,i}+1-l_{k-1,q})!}{(l_{k,i}-l_{k-1,q})!} = \prod_{i \leq q < k} (l_{k,i}-l_{k-1,q}+1).
\] 
When $t-1=k$ and $q=i$, we obtain the following noncancelling factors.
\[
  \prod_{1 \leq p \leq i} \frac{(l_{k+1,p}-(l_{k,i}+1))!}{(l_{k+1,p}-l_{k,i})!} = \prod_{1 \leq p \leq i}  \frac{1}{(l_{k+1,p}-l_{k,i})}.
\] 
Thus
\[
\frac{\hat{N}_1}{N_1} = \frac{\prod_{i \leq q < k} (l_{k,i}-l_{k-1,q}+1)}{\prod_{1 \leq p \leq i} (l_{k+1,p}-l_{k,i})}.
\]

Analyzing the ratios $\frac{\hat{D_1}}{D_1} $, $\frac{\hat{N_2}}{N_2} $, and $\frac{\hat{D_2}}{D_2} $ in a similar fashion leads to the following formulas.
\begin{align*}
\frac{\hat{D}_1}{D_1} &= \frac{\prod_{i < q < k+1} (l_{k,i}-l_{k,q}+1)}{\prod_{1 \leq p < i}  (l_{k,p}-l_{k,i})}\\
\frac{\hat{N}_2}{N_2} &= \frac{\prod_{i \leq q \leq k} (l_{k,i}-l_{k,q})}{\prod_{1 \leq p <i}  (l_{k,p}-l_{k,i}-1)}\\
  \frac{\hat{D}_2}{D_2} &= \frac{\prod_{i \leq q \leq k+1} (l_{k,i}-l_{k+1,q})}{\prod_{1 \leq p < i} (l_{k-1,p}-l_{k,i}-1)}
\end{align*}

Combining these formulas and simplifying yields the result.
\end{proof}

Combining Propositions \ref{prop: Molev action} and \ref{prop: ratio
  of lengths} yields the following.
\begin{proposition} \label{prop: matrix entries} \mbox{}
  \begin{enumerate}
  \item  The coefficient of $\frac{1}{\| P + \delta_{k,i}\|} e_{P + \delta_{k,i}}$ in $X_{\alpha_k} \frac{1}{\| P \|} e_P$ is 
  \[
    \sqrt{\frac{(-1)
      \displaystyle \prod_{q=1}^{k-1}(l_{k,i}-l_{k-1,q}+1)
      \displaystyle \prod_{q=1}^{k+1}
      (l_{k,i}-l_{k+1,q})}{\displaystyle
      \prod_{q=1}^{k}(l_{k,i}-l_{k,q}+1) \displaystyle \prod_{\substack{q=1,\\q\neq i}}^{k} (l_{k,i}-l_{k,q})}}.
\]
\item The coefficient of $\frac{1}{\| P - \delta_{k,i}\|} e_{P -
    \delta_{k,i}}$ in $Y_{\alpha_k} \frac{1}{\| P \|} e_P$ is 
    \[
    \sqrt{\frac{
      \displaystyle \prod_{q=1}^{k-1}(l_{k,i}-l_{k-1,q})
      \displaystyle \prod_{q=1}^{k+1}
      (l_{k,i}-l_{k+1,q}-1)}{\displaystyle
      \prod_{q=1}^{k}(l_{k,i}-l_{k,q}-1) \displaystyle \prod_{\substack{q=1,\\q\neq i}}^{k} (l_{k,i}-l_{k,q})}}.
\]
\end{enumerate}
\end{proposition}

Now we are ready to prove Proposition \ref{prop: negative transpose actions}.

\begin{proof}[Proof of Prop. \ref{prop: negative transpose actions}]
We prove Part 1. Part 2 is similar.

Let  $\mathscr{B} = \{ \frac{1}{\|P \|} e_{P}\} $ and let 
$\mathscr{B}^{*} = \{ \frac{(-1)^{\ell(P^{*})}}{\|  P^{*} \|}e_{P*}\}$. As in Definition \ref{def: support negative transpose actions} we write
  
\begin{align*}
  A &= [\rho_{\lambda} (X_{\alpha_k})]_{\mathscr{B}\leftarrow
        \mathscr{B}}\\
  B &= [\rho_{\lambda^{*}} (X_{\alpha_k})]_{\mathscr{B}^{*}\leftarrow
        \mathscr{B}^{*}}
\end{align*}

We introduce one more basis and matrix. Let$\mathscr{B}_{+}^{*} = \{ \frac{1}{\|  P^{*} \|}e_{P*}\}$, and let
\[
  B^{+} = [\rho_{\lambda^{*}} (X_{\alpha_k})]_{\mathscr{B}_{+}^{*}\leftarrow
        \mathscr{B}_{+}^{*}}
\]

Suppose $P$ and $P+\delta_{k,i}$ are valid patterns. Write $P' =
(P+\delta_{k,i})^{*}$. This prime notation will only be used in this
proof. Then $P'$ is related to $P^{*}$ by $P'+\delta_{k,k+1-i} = P^{*}$. Write $i' = k+1-i$. Suppose $P$ has index $I$ in $GT(\lambda)$, and $P+\delta_{k,i}$ has index $J$. Then the matrix entry $a_{I,J}$ is the  coefficient of $\frac{1}{\| P + \delta_{k,i}\|} e_{P + \delta_{k,i}}$ in $X_{\alpha_k} \frac{1}{\| P \|} e_P$, and the matrix entry $b_{J,I}^{+}$ is the coefficient of $\frac{1}{\| P^*\|} e_{P^*}$ in $X_{\alpha_k} \frac{1}{\| P' \|} e_{P'}$. We show first that $a_{I,J} = b_{J,I}^{+}$. Then, at the end, we show that changing from the basis $\mathscr{B}_{+}^{*}$ to the basis $\mathscr{B}^{*}$ gives the extra minus sign needed to get the negative transpose.

By Proposition \ref{prop: matrix entries}, we have
\begin{align*}
a_{I,J} &=    \sqrt{\frac{(-1)
      \displaystyle \prod_{q=1}^{k-1}(l_{k,i}-l_{k-1,q}+1)
      \displaystyle \prod_{q=1}^{k+1}
      (l_{k,i}-l_{k+1,q})}{\displaystyle
      \prod_{q=1}^{k}(l_{k,i}-l_{k,q}+1) \displaystyle \prod_{\substack{q=1,\\q\neq i}}^{k} (l_{k,i}-l_{k,q})}}\\
b_{J,I} &=    \sqrt{\frac{(-1)
      \displaystyle \prod_{q'=1}^{k-1}(l'_{k,i'}-l'_{k-1,q'}+1)
      \displaystyle \prod_{q'=1}^{k+1}
      (l'_{k,i'}-l'_{k+1,q'})}{\displaystyle
      \prod_{q'=1}^{k}(l'_{k,i'}-l'_{k,q'}+1) \displaystyle \prod_{\substack{q'=1,\\q'\neq i}}^{k} (l'_{k,i'}-l'_{k,q})}}
\end{align*}

First, we show that
\begin{equation} \label{eqn: goal 1}
\prod_{q=1}^{k-1}(l_{k,i}-l_{k-1,q}+1) = (-1)^{k-1}\prod_{q'=1}^{k-1}(l'_{k,i'}-l'_{k-1,q'}+1)
\end{equation}
via the map $q' \mapsto k-q$.

We rewrite the left hand side of (\ref{eqn: goal 1}) as follows.
\begin{align*}
\prod_{q=1}^{k-1}(l_{k,i}-l_{k-1,q}+1) &= \prod_{q=1}^{k-1}((x_{k,i}-i+1)-(x_{k-1,q}-q+1)+1)\\ 
&= \prod_{q=1}^{k-1}(x_{k,i}-x_{k-1,q}+q-i+1). 
\end{align*}
Next, we rewrite the right hand side of (\ref{eqn: goal 1}).
\begin{align*}
\prod_{q'=1}^{k-1}(l'_{k,i'}-l'_{k-1,q'}+1) &= \prod_{q'=1}^{k-1}((y_{k,i'}-i'+1)-(y_{k-1,q'}-q'+1)+1) \\
&= \prod_{q'=1}^{k-1}(y_{k,i'}-y_{k-1,q'}+(q'-i')+1) \\
&= \prod_{q'=1}^{k-1}((x_{n,1}-(x_{k,k+1-i'}+1))-(x_{n,1} - x_{k-1,k-q'})+(q'-i')+1) \\
&= \prod_{q'=1}^{k-1}(-x_{k,k+1-i'}+x_{k-1,k-q'}+(q'-i')) \\
&= \prod_{q=1}^{k-1}(-x_{k,i}+x_{k-1,q}-q+i-1). \\
&= (-1)^{k-1}\prod_{q=1}^{k-1}(x_{k,i}-x_{k-1,q}+q-i+1).    
\end{align*}
This proves (\ref{eqn: goal 1}).

Similar arguments lead to the following equations.
\begin{align*}
  \prod_{q=1}^{k+1} (l_{k,i}-l_{k+1,q}) &= (-1)^{k+1} \prod_{q'=1}^{k+1}(l'_{k,i'}-l'_{k+1,q'}) &\text{ via } q' \mapsto k+2-q;\\
  \prod_{q=1}^{k}(l_{k,i}-l_{k,q}+1) &=  (-1)^{k-1}\prod_{\substack{q'=1,\\q'\neq i}}^{k} (l'_{k,i'}-l'_{k,q}) &\text{ via } q' \mapsto k+1-q;\\
  \prod_{\substack{q=1,\\q\neq i}}^{k} (l_{k,i}-l_{k,q}) &=  (-1)^{k-1}\prod_{q'=1}^{k}(l'_{k,i'}-l'_{k,q'}+1) &\text{ via } q' \mapsto k+1-q.
\end{align*} 

Combining these formulas, we have $a_{I,J} = b_{J,I}^{+}$.

Finally, we argue that $b_{J,I} = -b_{J,I}^{+}.$ The patterns $P'$ and $P*$ satisfy $wt(P^{*}) = wt(P') + \alpha_k$. Thus the levels $\ell(P^{*})$ and $\ell(P')$ have opposite parity, and hence  $b_{J,I} = -b_{J,I}^{+}.$
\end{proof}

\subsection{Rationality of the coefficients}

It turns out that the expressions $\| P \| \|P^{*} \|$ appearing in
the formula for the invariant are rational numbers. This is useful
when computing, as it allows us to do exact calculations over
$\mathbb{Q}$ rather than approximate calculations over $\mathbb{R}$.

\begin{proposition}
\[
  \| P \| \|P^{*} \| = \prod_{t=2}^{n} \frac{\displaystyle \prod_{1
      \leq p < q \leq t}(l_{t,p}-l_{t,q}-1)!}{\displaystyle \prod_{1 \leq p \leq q < t}(l_{t-1,p}-l_{t-1,q})!}
\]
\end{proposition}
\begin{proof}
  Write $\| P \| = \frac{N_1 N_2}{D_1 D_2}$ and $\| P^{*} \| = \frac{N_1^{*} N_2^{*}}{D_1^{*} D_2^{*}}$. We will show that $N_1^{*} = D_2$, $D_1^{*} = D_1$, $N_2^{*}=N_2$, and $D_2^{*} =N_1$. Thus
  \[
    \| P \| \|P^{*} \| = \sqrt{\frac{N_2^2}{D_1^2}} = \left| \frac{N_2}{D_1} \right| = \frac{N_2}{D_1}.
\] 
  
Consider
\[
  D_2 = \displaystyle\prod_{t=2}^{n} \prod_{1 \leq p < q \leq t} (l_{t-1,p} - l_{t,q}-1)!
\]
 and
\[
N_1^{*} = \displaystyle\prod_{t=2}^{n}\prod_{1 \leq p^{*} \leq q^{*} < t} (l^{*}_{t,p^{*}} - l^{*}_{t-1,q^{*}})! 
\]

We have
\begin{align*}
  (l_{t-1,p} - l_{t,q}-1) &= x_{t-1,p}-p+1 - (x_{t,q}-q+1)-1\\
  &=x_{t-1,p}-x_{t,q} -p+q-1.
\end{align*}

Applying the map $p = t-q^{*}$ and $q = t+1-p^{*}$ we have
\begin{align*}
  l^{*}_{t,p^{*}} - l^{*}_{t-1,q^{*}} &= l^{*}_{t,p^{*}} - l^{*}_{t-1,q^{*}} \\
  &= y_{t,p^{*}}-p^{*}+1-(  y_{t-1,q^{*}} -q^{*}+1)\\
  &= (x_{n,1} - x_{t,t+1-p^{*}})-(x_{n,1} - x_{t-1,t-q^{*}})+(q^{*}-p^{*})\\
  &= -x_{t,q}+x_{t-1,p}+q-p-1.
\end{align*}
Thus, $N_1^{*} = D_2$.

Similar arguments establish the remaining equations $D_1^{*} = D_1$, $N_2^{*}=N_2$, and $D_2^{*} =N_1$. 

\end{proof}

\section*{References}
\bibliographystyle{amsplain}

\end{document}